\documentclass[10pt,twocolumn]{article}

\usepackage[utf8]{inputenc} % allow utf-8 input
\usepackage[T1]{fontenc}    % use 8-bit T1 fonts
\usepackage{amsmath,amssymb,amsthm} % pour les maths
\usepackage{url}            % simple URL typesetting
\usepackage{booktabs}       % professional-quality tables
\usepackage{amsfonts}       % blackboard math symbols
\usepackage{nicefrac}       % compact symbols for 1/2, etc.
\usepackage{microtype}      % microtypography
\usepackage{xcolor}                 % Pour mettre des couleurs
\usepackage[color=orange!30]{todonotes}
\usepackage{pgfplots} 
\usepackage{subcaption} 
\usepackage{dsfont}                 % pour la commande mathds
\usepackage{cancel}                 % pour barrer des formules
\usepackage{enumitem}               % pour mieux gérer les listes (espacements,...)
\usepackage[american]{babel}
\usepackage{csquotes}
\usepackage{enumitem}
\usepackage{natbib}
\usepackage{authblk}
\usepackage[a4paper,includeheadfoot,top=1in, bottom=1.25in, left=.8in, right=.8in]{geometry}
\usepackage{silence}% Filter out unwanted warnings and error messages
\let\chapter\section
\usepackage[ruled]{algorithm2e}     % Pour définir l'algo

\renewcommand{\thefootnote}{$\dagger$} 

\graphicspath{{img/}}

\allowdisplaybreaks
\newtheorem{theorem}{Theorem}

\newtheorem{lemma}[theorem]{Lemma}

\theoremstyle{remark}
\newtheorem{remark}{Remark}[section]

\def\1{\ensuremath{\mathrm{1}\hspace{-.35em} \mathrm{1}}} % indicatrice

\def\E{\mathbb{E}}

\def\R{\mathbb{R}}

\def\cF{\mathcal{F}}

\def\cB{\mathcal{B}}

\def\cE{\mathcal{E}}
\def\cG{\mathcal{G}}

\def\cN{\mathcal{N}}
\def\cO{\mathcal{O}}
\def\cS{\mathcal{S}}
\renewcommand{\hat}{\widehat}
\renewcommand{\epsilon}{\varepsilon}
\renewcommand{\leq}{\leqslant}
\renewcommand{\geq}{\geqslant}

\def\mystrut(#1,#2){\vrule height #1pt depth #2pt width 0pt} 

\DeclareMathOperator*{\argmin}{arg\,min}
\DeclareMathOperator{\Reg}{Reg}
\DeclareMathOperator{\Risk}{Risk}
\DeclareMathOperator{\Err}{Err}

\makeatletter
\newcommand{\pushright}[1]{\ifmeasuring@#1\else\omit\hfill$\displaystyle#1$\fi\ignorespaces}
\newcommand{\pushleft}[1]{\ifmeasuring@#1\else\omit$\displaystyle#1$\hfill\fi\ignorespaces}
\makeatother

\begin{document}

\title{Sparse Accelerated Exponential Weights}

\author[1]{Pierre~Gaillard\thanks{pierre@gaillard.me}}
\author[1,2]{Olivier~Wintenberger\thanks{wintenberger@math.ku.dk}}
\affil[1]{University of Copenhagen, Denmark}
\affil[2]{Sorbonne Universités, UPMC Univ Paris 06, F-75005, Paris, France}

% \author{
%   Pierre~Gaillard\\
%   \texttt{pierre@math.ku.dk}
%   \and
%   Olivier~Wintenberger\\
%   \texttt{wintenberger@math.ku.dk}
% }

% \aistatsaddress{ 
%  University of Copenhagen, Denmark \And 
%  University of Copenhagen, Denmark \\
%   Sorbonne Universités, UPMC Univ Paris 06 LSTA, FRANCE} ]

    \maketitle

\begin{abstract}
We consider the stochastic optimization problem where a convex function is minimized observing recursively the gradients.  We introduce SAEW, a new procedure that accelerates exponential weights procedures with the slow rate $1/\sqrt T$ to procedures achieving the fast rate $1/T$. Under the strong convexity of the risk, we achieve the optimal rate of convergence for approximating sparse parameters in $\R^d$. The acceleration is achieved by using successive  averaging steps in an online fashion. The procedure also produces sparse estimators thanks to additional hard threshold steps.
\end{abstract}

\section{Introduction}

Stochastic optimization procedures have encountered more and more success in the past few years. This common framework includes machine learning methods minimizing the empirical risk.  \citet{LeCunBottou2004}  emphasized the utility of  Stochastic Gradient Descent (SGD) procedures compared with batch procedures; the lack of accuracy in the optimization is balanced by the robustness of the procedure to any random environment. \citet{Zinkevich2003} formalized this robustness property by proving a $d/\sqrt T$ rate of convergence  in any possible convex environment for a $d$-dimensional parametric bounded space. This rate is optimal with no additional condition. However, under strong convexity of the risk,  accelerated SGD procedures achieve the fast rate $d/T$, that is also optimal \citet{Agarwal2012}. One of the most popular acceleration procedure is obtained by a simple averaging step, see \citet{PolyakJudisky1992} and \citet{BachMoulines2013}. Other robust and adaptive procedures using exponential weights have  been studied in the setting of individual sequences by \citet{Cesa-BianchiLugosi2006}. The link with the stochastic optimization problem has been done in \citet{KivinenWarmuth1997}, providing in the $\ell_1$-ball algorithms with an optimal logarithmic dependence on the dimension $d$ but a slow rate $1/\sqrt T$. The fast rate $\log(T)$ on the regret has been achieved in some strongly convex cases as in Theorem 3.3 of \citet{Cesa-BianchiLugosi2006}. Thus, the expectation of the risk of their averaging, studied under the name of progressive mixture rule by \citet{Catoni2004}, also achieves the fast rate $\log(T)/T$. However, progressive mixture rules do not achieve the fast rate with high probability, see \citet{Audibert2008} and their complexity is prohibitive (of order $d^T$). The aim of this paper is to propose an efficient acceleration of  exponential weights procedures that achieves the fast rate $1/T$ with high probability. 

In parallel, optimal rates of convergence for the risk were provided by \citet{BuneaTsybakovWegkamp2007} in the sparse setting. When the optimal parameter $\theta^\ast$ is of dimension $d_0=\|\theta^\ast\|_0$ smaller than the dimension of the parametric space $d$, the optimal rate of convergence is
${d_0\log(d)}/T$. Such fast rates can be achieved for polynomial time algorithm only up to the multiplicative factor $\alpha^{-1}$ where $\alpha$ is the strong convexity constant of the risk, see \citet{Zhang2014}.  For instance, the Lasso procedure achieves this optimal rate for least square linear regression, see Assumption (A3) (implying strong convexity of the risk) of \citet{BuneaTsybakovWegkamp2007}. Other more robust optimal batch procedures such as $\ell_0$ penalization or exploration of the parametric space suffer serious complexity drawbacks and are known to be NP-hard.  Most of the stochastic algorithms do not match this rate, with the exception of SeqSEW (in expectation only), see \citet{Gerchinovitz2011a}. As the strong convexity constant $\alpha$ does not appear in the bounds of \citet{Gerchinovitz2011a}, one suspects that the algorithm is NP-hard. 

%In the specific case $d_0=1$, the acceleration of exponential weights called BOA of \citet{Wintenberger2014} achieves the fast rate with high probability.

The aim of this paper is to provide the first acceleration of  exponential weights procedures achieving the optimal rate of convergence ${d_0\log(d)}/(\alpha T)$ in the identically and independently distributed (i.i.d.) online optimization setting with sparse solution~$\theta^\ast$.  The acceleration is obtained by localizing the exponential weights around their averages in an online fashion.  The idea is that the averaging alone suffers too much from the exploration of the entire parameter space. The sparsity is achieved by an additional hard-truncation step, producing sparse approximations of the optimal parameter~$\theta^\ast$. The acceleration procedure is not computationally hard as its complexity is $\cO(dT)$. We obtain theoretical optimal bounds on the risk similar to the Lasso for random design, see \citet{BuneaTsybakovWegkamp2007}. We also obtain optimal bounds on the cumulative risk of the exploration of the parameter space.\\

The paper is organized as follows. After some preliminaries in Section \ref{sec:prel}, we present our acceleration procedure and we prove that it achieves the optimal rate of convergence in Section \ref{sec:known_parameters}. We refine the constants for least square linear regression  in Section \ref{sec:squareloss}. Finally, we give some simulations in Section \ref{sec:as}.

\section{Preliminaries}\label{sec:prel}

    We consider a sequence $\ell_t:\R^d\to \R, t\geq 1$ of i.i.d. random loss functions. We define the instantaneous risk as $\E[\ell_t]:\theta\mapsto\E[\ell_t(\theta)]$\footnote{Because the losses are i.i.d, the risk does not depend on $t\geq 1$. However,  we still use the time index in the notation to emphasize that a quantity indexed by $s \geq 1$ cannot depend on $\ell_t$ for any $t > s$. The notation $\E[\ell_t](\hat \theta_{t-1})$ denotes  $\E\big[\ell_t(\hat \theta_{t-1})\big|\ell_1,\dots,\ell_{t-1}\big]$.}. 
    We assume that the risk is $(2\alpha)$-strongly convex, i.e., for all $\theta_1,\theta_2\in \R^d$
    \begin{multline}
    \label{ass:strong_convexity}
        \tag{SC} \raisetag{3em}
        \E\big[\ell_t(\theta_1) - \ell_t(\theta_2)\big] \leq \E\big[\nabla\ell_t(\theta_1)\big]^\top (\theta_1-\theta_2) \\
             - \alpha \big\|\theta_1-\theta_2 \big\|_2^2\,.
    \end{multline}
    The (unique) risk minimizer in $\R^d$ is denoted $\theta^\ast$ and its effective dimension is $\|\theta^\ast\|_0 \leq d_0$. We insist on the fact that the strong convexity  is only required on the risk and not on the loss function. This condition is satisfied for many non strongly convex loss functions such as the quantile loss (see Section \ref{sec:as}) and necessary  to obtain fast rates of convergence (see \cite{Agarwal2012}). 

     \paragraph{Online optimization setting} For each $t\geq 1$, we provide two parameters $(\smash{\hat \theta_{t-1},\tilde \theta_{t-1}) \in \R^d\times \R^d}$ having observed the past gradients of the first parameter $\smash{\nabla \ell_s(\hat \theta_{s-1}) \in \R^d}$ for $s\leq t-1$ only.  
     
Our aim is to  provide high-probability upper-bounds on the cumulative excess risk (also called cumulative risk for simplicity) of the sequence~$(\hat \theta_{t-1})$ and on the instantaneous excess risk of~$\tilde \theta_{t-1}$:
     \begin{itemize}
        \item \emph{Cumulative risk:} the online exploration vs. exploitation problem aims at minimizing the cumulative risk of the  sequence $\smash{(\hat \theta_{t-1})}$ defined as
    \begin{equation}
        \label{eq:cumrisk}
        \Risk_{1:T}(\hat \theta_{0:(T-1)}) := \sum_{t=1}^T \Risk(\hat \theta_{t-1}) \,,
    \end{equation}
    where $\Risk(\theta) := \E[\ell_t](\theta) - \E[\ell_t](\theta^\ast)$ is the instantaneous excess risk.
     This goal is useful in a predictive scenario when the observation of $\nabla \ell_t(\hat \theta_{t-1})$ comes at the cost of $\smash{\Risk(\hat \theta_{t-1})}$.
        \item \emph{Instantaneous excess risk}: simultaneously, at any time $t\geq 1$, we provide an estimator $\tilde \theta_{t-1}$ of $\theta^\ast$ that minimizes the instantaneous risk. This problem has been widely studied in statistics and the known solutions are mostly batch algorithms. Under the strong convexity of the risk, a small instantaneous risk ensures in particular that $\tilde \theta_{t-1}$ is close in $\ell_2$-norm to the true parameter $\theta^\ast$ (by Lemma~\ref{lem:strong_convex}, Appendix~\ref{app:lemme1}).
    \end{itemize}

    To make a parallel with the multi-armed bandit setting, minimizing the cumulative risk is related to minimizing the cumulative regret. In contrast, the second goal is related to simple regret (see~\cite{Bubeck2009}): the cost of exploration only comes in terms of resources (time steps $T$) rather than of costs depending on the exploration. 

    By convexity of the risk, the averaging
    $
        \smash{\bar \theta_{T-1}} :=$ $\smash{ (1/T) \sum_{t=1}^{T} \hat \theta_{t-1}}
    $
    has an instantaneous risk upper-bounded by the cumulative risk
    \begin{equation}
        \label{eq:jensen}
        \Risk(\bar \theta_{T-1})  \leq \Risk_{1:T}(\hat \theta_{0:(T-1)}) / T \,.
    \end{equation}
    Therefore, upper bounds on the cumulative risk lead to upper bounds on the instantaneous risk for $\tilde \theta_{T-1} = \bar \theta_{T-1}$. However, we will provide another solution to build $\tilde \theta_{T-1}$ with better guarantees than the one obtained by~\eqref{eq:jensen}.

    On the contrary, since each $\tilde \theta_{t-1}$ minimizes the instantaneous risk at time $t$, it is tempting to use them in the exploration vs. exploitation problem. However, it is impossible in our setting as the parameters $(\tilde \theta_{t})$ are constructed upon the observation of the gradients $\nabla \ell_s(\hat \theta_{s-1})$, $s<t$.  Remark that our bounds on the cumulative risk are optimal as of the same order than $\sum_{t=1}^T \Risk(\tilde \theta_{t-1})$.

    Our main contribution (see Theorems~\ref{thm:SAEW} and~\ref{thm:cumulativerisk}) is to introduce a new acceleration procedure that simultaneously ensures (up to loglog terms) both optimal risk for $\smash{\tilde \theta_{t-1}}$ and optimal cumulative risk for $\smash{(\hat \theta_{t-1})}$. Up to our knowledge, this is the first polynomial time online procedure that recovers the minimax rate obtained in a sparse strongly convex setting. 
    Its instantaneous risk achieves the optimal rate of convergence 
    \begin{equation}
    \min\left\{\frac{B^2d_0\log(d)}{\alpha T},UB\sqrt{\frac{\log(d)}T}\right\},
    \label{eq:rate}
    \end{equation}
    where $B\geq \sup_{\theta: \|\theta\|_1\leq 2U} \|\nabla\ell_t(\theta)\|_\infty$ is an almost sure bound on the gradients,
 \begin{equation}
        \label{ass:small_diameter}
        \big\|\theta^\ast\|_1 \leq U  \quad \text{and} \quad \big\|\theta^\ast\|_0 \leq d_0 \,.
    \end{equation}

    For least square linear regression (see Theorem~\ref{thm:squareloss}), $B^2$ is replaced in~\eqref{eq:rate} with a term of order $\sigma^2:= \E[\ell_t(\theta^\ast)]$. In the batch setting, the Lasso achieves a similar rate under the slightly stronger Assumption (A3) of \citet{BuneaTsybakovWegkamp2007}.

    \section{Acceleration procedure for known parameters}
    \label{sec:known_parameters}

     We propose SAEW (described in Algorithm~\ref{alg:SAEW}) that depends on the parameters $(d_0,\alpha,U,B)$ and performs an optimal online optimization in the $\ell_1$ ball of radius $U$. SAEW accelerates a convex optimization subroutine (see Algorithm~\ref{alg:subroutine}). If the latter achieves a slow rate of convergence on its cumulative regret, SAEW achieves a fast rate of convergence on its cumulative and instantaneous risks. We describe first what is expected from the subroutine.

    \subsection{Convex optimization in the $\ell_1$-ball with a slow rate of convergence}

Assume that a generic subroutine (Algorithm~\ref{alg:subroutine}), denoted by $\cS$,  performs online convex optimization into the $\ell_1$-ball $\cB_1 \big(\theta_{\mathrm{center}},\epsilon \big) := \big\{\theta \in \R^d: \|\theta - \theta_{\mathrm{center}}\|_1 \leq \epsilon\big\}$ of center $\theta_{\mathrm{center}}\in \R^d$ and radius $\epsilon >0$. Centers and radii will be settled online thanks to SAEW.

    \begin{algorithm}
    \caption{Subroutine $\cS$: convex optimization in $\ell_1$-ball}
    \label{alg:subroutine}
    \begin{minipage}{0.45\textwidth}
    {\bfseries Parameters:} $B>0$, $t_\mathrm{start} > 0$, $\theta_{\mathrm{center}} \in \R^d$  and $\epsilon >0$.\\[5pt]
    For each  $t = t_{\mathrm{start}}, t_{\mathrm{start}}+1, \dots$, 
    \begin{itemize}
        \item predict $\hat \theta_{t-1} \in \cB_1(\theta_{\mathrm{center}},\epsilon)$ (thanks to some online gradient procedure)
        \item suffer loss $\ell_t\big(\hat \theta_{t-1}\big) \in \R$ and observe the gradient $\nabla \ell_t\big(\hat \theta_{t-1}\big) \in \R^d$
    \end{itemize}
    \end{minipage} 
    \end{algorithm}

     We assume that the subroutine $\cS$ applied on any sequence of convex sub-differentiable losses $(\ell_{t})_{t \geq t_{\mathrm{start}}}$ satisfies the following upper-bound on its cumulative  regret: for all $t_{\mathrm{end}} \geq t_{\mathrm{start}}$ and for all $\theta \in \cB_1\big(\theta_{\mathrm{center}},\epsilon\big)$
    \begin{equation}
        \label{eq:reg_online_optimization}
        %\Reg\big(\hat \theta_{t_\mathrm{start}:t_\mathrm{end}}\big) := 
        \sum_{t=t_{\mathrm{start}}}^{t_{\mathrm{end}}}  \ell_t(\hat \theta_{t-1}) -   \ell_t(\theta) 
        \leq a \epsilon \sqrt{\sum_{t_{\mathrm{start}}}^{t_{\mathrm{end}}} \big\|\nabla \ell_t(\hat \theta_{t-1})\big\|_\infty^2 } 
        + b\epsilon B\,,
    \end{equation}
    for some non-negative constants $a,b$ that may depend on the dimension $d$.

    Several online optimization algorithms do satisfy the regret bound~\eqref{eq:reg_online_optimization} while being totally tuned, see for instance~\citet[Corollary 2.1]{Gerchinovitz2011} or \citet{CesaBianchiMansourStoltz2007,GaillardStoltzEtAl2014,Wintenberger2014}. 
    The regret bound is satisfied for instance with\renewcommand{\thefootnote}{$\ddagger$} \footnote{As in the rest of the paper, the sign $\lesssim$ denotes an inequality which is fulfilled up to multiplicative constants.} $\smash{a \lesssim \sqrt{\log d}}$ and $b \lesssim \log d$ by a well online-calibrated Exponentiated Gradient (EG) forecaster combining the corners of $\smash{\cB_1\big(\theta_{\mathrm{center}},\epsilon\big)}$. This logarithmic dependence on the dimension is crucial here and possible because the optimization is performed in the $\ell_1$-ball.  SGD optimizing in the $\ell_2$-ball, such as RDA of~\citet{Xiao2010}, suffer a linear dependence on $d$. Therefore, they cannot be used as subroutines.

    The regret bound yields the slow rate of convergence $\smash{\cO\big(\sqrt{(\log d)(t_{\mathrm{end}} - t_{\mathrm{start}})}\big)}$ (with respect to the length of the session) on the cumulative risk. Our acceleration procedure provides a generic method to also achieve a fast rate under sparsity.

    \subsection{The acceleration procedure}

\begin{algorithm}[th!]
    \caption{SAEW}
    \label{alg:SAEW}
    \begin{minipage}{.45\textwidth}
    {\bfseries Parameters:} $d_0 \geq 1$, $\alpha >0$, $U > 0$, $B>0$, $\delta>0$ and a subroutine $\cS$ that satisfies~\eqref{eq:reg_online_optimization} \\[5pt]
    {\bfseries Initialization:} $t_0 = t = 1, \epsilon_0 = U$ and $\bar \theta_{0} = 0$ \\[5pt]
    For each  $i = 0,1,\dots$ 
    \begin{itemize}
      \item define $\smash{[\bar \theta_{t_{i}-1}]_{d_0}}$ by rounding to zero the $d-d_0$ smallest coefficients of $\smash{\bar \theta_{t_{i}-1}}$
      \item start a new instance $\cS_i$ of the subroutine $\cS$ with parameters $t_{\mathrm{start}} = t_i$, $\smash{\theta_{\mathrm{center}} = [\bar \theta_{t_i-1}]_{d_0}}$, $\epsilon = U 2^{-i/2}$ and $B$, 
      \item for $t = t_i, t_i +1, \dots$ and while $\epsilon_{t-1} > U 2^{-(i+1)/2}$
      \begin{itemize}
        \item forecast $\hat \theta_{t-1}$ by using the subroutine $\cS_i$
        \item observe $\nabla \ell_t(\hat \theta_{t-1})$
        \item update the bound
        \[
            \Err_t := a_i' \sqrt{\sum_{s=t_i}^{t} \big\|\nabla\ell_s(\hat \theta_{s-1})\big\|_\infty^2} + b_i' B
        \]
        with $a_i'$ and $b_i'$ resp. defined in~\eqref{eq:defa} and~\eqref{eq:defb}.
        \item update the confidence radius 
          \[
            \epsilon_{t} := 2\sqrt{\frac{2d_0 U 2^{-i/2}}{\alpha(t-t_i+ 1)}  \Err_t}
          \] 
        \item update the averaged estimator 
        \[
           \textstyle \bar \theta_{t} :=  (t-t_i+1)^{-1} \sum_{s=t_i}^t \hat \theta_{s-1}
        \]
        \item update the estimator
        \[
            \tilde \theta_{t} := \bar \theta_{\argmin_{0\leq s \leq t} \epsilon_s}
        \]
      \end{itemize}
      \item stop the instance $\mathcal{S}_i$ and define $t_{i+1} := t+1$ 
    \end{itemize}
    \end{minipage}
    \end{algorithm}

    Our acceleration procedure (SAEW, described in Algorithm~\ref{alg:SAEW}) performs the subroutine $\cS$ on sessions of adaptive length  optimizing in exponentially decreasing $\ell_1$-balls. The sessions are indexed by $i \geq 0$ and denoted $\cS_i$. The algorithm defines in an online fashion a sequence of starting times $1 = t_0 < t_1 < \dots $ such that the instance $\cS_{i}$ is used to perform predictions between times $t_{\mathrm{start}}=t_i$ and $t_{\mathrm{end}}=t_{i+1}-1$. The idea is that our accuracy in the estimation of $\theta^\ast$ increases over time so that $\cS_i$ can be a localized optimization subroutine in a small ball $\smash{\cB_1\big([\bar \theta_{t_{i}-1}]_{d_0}, U 2^{-i/2}\big)}$ around the current sparse estimator $[\bar \theta_{t_{i}-1}]_{d_0}$ of $\theta^\ast$ at time $t_i$, see Algorithm~\ref{alg:SAEW} for the definition of $[\bar \theta_{t_{i}-1}]_{d_0}$. 

    The cumulative risk suffered during each session will remain constant: the increasing rate $\smash{\big(\sum_{t_i}^{t_{i+1}-1} \big\|\nabla \ell_t(\hat \theta_{t-1})\big\|_\infty^2 \big)^{1/2}} \leq B \sqrt{t_{i+1}-t_i}$ due to the length of the session (see Equation \eqref{eq:reg_online_optimization}) will be shown to be of order $2^{i/2}$. But it will be offset by the decreasing radius $\epsilon=U2^{-i/2}$. 

    By using a linear-time subroutine $\cS$, the global time and storage complexities of SAEW are also $\cO(dT)$.

    \smallskip
    Our main theorem is stated below. It controls the excess risk of the instantaneous estimators of SAEW. The proof is deferred to Appendix~\ref{proof:SAEW}.

    \begin{theorem}
        \label{thm:SAEW}
%        Let $d_0 \geq \|\theta^\ast\|_0$, $U\geq \|\theta^\ast\|_1$ and $0<\delta<1$. Assume that 
%        $
%            B \geq \max_{\theta:\|\theta\|_1\leq 2U} \|\nabla \ell_t(\theta)\|_\infty
%        $
%        almost surely. 
Under Assumption~\eqref{ass:strong_convexity}, SAEW satisfies with probability at least $1-\delta$, $0<\delta<1$, for all $T\geq 1$
        \begin{align*}
            \Risk\big(\tilde \theta_{T} \big) 
             \leq 
            \min & \bigg\{  UB \left(a'\sqrt{\frac{2}{T}} + \frac{4b'}{T}\right) + \frac{\alpha U^2}{8d_0T}, \\
            & {\frac{d_0 B^2}{\alpha} \left(\frac{2^7 a'^2}{T} + \frac{2^{11} b'^2}{T^2}\right) + \frac{2\alpha U^2}{d_0 T^2}
            \bigg\}\,,}
        \end{align*}
        where $a' = a + \sqrt{6 \log(1 + 3 \log T) - 2 \log \delta }$ and $b' = b  +  1/2 + 3 \log(1+3\log t) - \log \delta$.
    \end{theorem}
    \begin{remark}
    Using EG as the subroutines, the main term of the excess risk becomes of order 
    \begin{equation}   
        \label{eq:ratethetatilde}
       \Risk\big(\tilde \theta_{T} \big)  =  \cO_T\bigg(\frac{d_0B^2}{\alpha T}\log \Big(\frac{d \log T}{\delta}\Big)\bigg) \,.
    \end{equation}
    \end{remark}

    \begin{remark}  From the strong convexity assumption, Theorem~\ref{thm:SAEW} also ensures that, with probability $1-\delta$, the  estimator $\tilde \theta_T$ is close enough to $\theta^\ast$:
    \[
        \big\|\tilde  \theta_T - \theta^\ast\big\|_2 \lesssim \frac{ \sqrt{d_0} B}{\alpha  \sqrt{T} } \sqrt{a'^2\log_2 T + \frac{b'^2}{T}+ \frac{\alpha U^2}{d_0 T}} \,.
    \]  
    \end{remark}

    \begin{theorem}
        \label{thm:cumulativerisk}
        Under the assumptions and the notation of Theorem~\ref{thm:SAEW}, the cumulative risk of SAEW is upper-bounded with probability at least $1-\delta$ as
        \begin{multline*}
          \Risk_{1:T}(\hat\theta_{0:(T-1)}) 
            \leq  \min\bigg\{ 4UB(a' \sqrt{T} + b' + 1), \\
            \frac{ 2^5 d_0 B^2}{\alpha } a'^2 \log_2 T + 4UB(1+b') + \frac{\alpha U^2 }{8 d_0} \bigg\} \,.
        \end{multline*}
    \end{theorem}

    \begin{remark} 
    Using EG as the subroutines, we get a cumulative risk of order
    \[
       \Risk_{1:T}(\hat\theta_{0:(T-1)})   =  \cO_T\bigg(\frac{d_0B^2}{\alpha T}\log \Big(\frac{d \log T}{\delta}\Big)\log T\bigg) \,.
    \]
    The averaged cumulative risk bound has an additional factor $\log T$ in comparison to the excess risk of $\tilde \theta_T$. This logarithmic factor is unavoidable. Indeed, at time $t$, the rate stated in Equation~\eqref{eq:ratethetatilde} is optimal for any estimator. An optimal rate for the cumulative risk can thus be obtained by summing this rate of order $\cO(1/t)$ over $t$ introducing the log factor.
    \end{remark}

    \begin{remark} \label{rem:unbounded_gradients}
    Adapting Corollary 13 of \citet{Gerchinovitz2011a},  the  boundedness of  $\nabla \ell_t$ can be weakened to unknown $B$ under the subgaussian condition. The price of this adaptation is a multiplicative factor of order $\log(dT)$ in the final bounds. 
    \end{remark}

\begin{remark}
Using the strong convexity property, the averaging of SAEW has much faster rate ($\log T/T$ on the excess risk) than the averaging of the EG procedure itself (only slow rate $1/\sqrt T$ with high probability, see \cite{Audibert2008}). But the last averaging $\tilde \theta_T$ achieves the best rate overall. Also note the difference of the impact of the $\ell_1$-ball radius $U$ on the rates: for the overall average $\bar\theta_T$ it is $U^2/T$ whereas it is $U^2/T^2$ for the last averaging $\tilde \theta_T$. On the contrary to the overall averaging, the last averaging forgets the cost of the exploration of the initial $\ell_1$-ball.\end{remark}

    \section{Square linear regression}
    \label{sec:squareloss}

    Consider the common least square linear regression  setting. Let $(X_t,Y_t)$, $t\geq 1$ be i.i.d. random pairs taking values in $\R^d \times \R$. For simplicity, we assume that $\|X_t\|_\infty \leq X$ and $|Y_t|\leq Y$ almost surely for some  constants $X,Y>0$. We aim at estimating linearly the conditional mean of $Y_t$ given $X_t$, by approaching $\theta^\ast = \argmin_{\theta \in \R^d} \E\big[(Y_t-X_t^\top \theta)^2\big]$. Notice that the strong convexity of the risk is equivalent to the positivity of the covariance matrix of $X_t$ as $\smash{\alpha \leq \lambda_{\min}\big(\E\big[X_tX_t^\top]\big)}$, where $\lambda_{\min}$ is the smallest eigenvalue.
    
    Applying the previous general setting to the square loss function 
      $
          \ell_t: \theta \mapsto (Y_t -  X_t^\top \theta)^2 \,,
      $
    we get the following Theorem~\ref{thm:squareloss}. It improves upon Theorem~\ref{thm:SAEW} the factor $B^2$ in the main term into a factor $X^2 \sigma^2$, where $\sigma^2 := \E\big[(Y_t - X_t^\top \theta^\ast)^2\big]$ is the expected loss of the best linear predictor. This is achieved without the additional knowledge of $\sigma^2$.
    The proof of the theorem is highly inspired from the one of Theorem~\ref{thm:SAEW} and is deferred to Appendix~\ref{proof:squareloss}.

    \begin{theorem}
        \label{thm:squareloss}
        SAEW tuned with $B = 2X\big(Y+2XU\big)$ satisfies with probability at least $1-\delta$ the bound
    \begin{multline*}
         \hspace*{-2ex} \Risk(\tilde \theta_T) \lesssim    \min \bigg\{  U X \left(\frac{ \sigma a'}{\sqrt{T}} + \frac{ (Y+XU) c '}{T}\right) + \frac{\alpha U^2}{d_0T}, \\
             {\frac{ X^2 d_0 }{\alpha} \left(\frac{\sigma^2 a'^2}{T} + \frac{(Y+XU)^2 c'^2}{T^2}\right) + \frac{\alpha U^2}{d_0 T^2}
            \bigg\}\,,}
    \end{multline*}
for all $T\geq 1$, where $a' \lesssim a + \sqrt{\log(1/\delta) + \log \log T}$ and $b' \lesssim b + \log(1/\delta) + \log \log T$.
    \end{theorem}

    \begin{remark}
    Using a well-calibrated EG for the subroutines, the main term of the excess risk is of order 
    \[
       \Risk\big(\tilde \theta_{T} \big)  =  \cO_T\bigg(\frac{d_0X^2\sigma^2}{\alpha T}\log \Big(\frac{d \log T}{\delta}\Big)\bigg) \,.
    \]
    \end{remark}

    \begin{remark}
    Similarly to Remark~\ref{rem:unbounded_gradients}, if $(X_t,Y_t)$ are subgaussian only (and not necessary bounded), classical arguments show that Theorem~\ref{thm:squareloss} still holds with $X$ of order $\cO(\log (dT))$ and $Y = \cO(\log T)$. 
    \end{remark}

    \begin{remark}
    The improvement from Theorem~\ref{thm:SAEW} to Theorem~\ref{thm:squareloss} (i.e., replacing $B$ with $X^2\sigma^2$ in the main term) is less significant if we apply it to the cumulative risk (Theorem~\ref{thm:cumulativerisk}). This would improve $B^2 \log T$ to $B^2 + X^2\sigma^2 \log T$ and thus lead to a bound on the cumulative risk of order $\cO(d_0 \sigma^2 \log (T)/\alpha )$.
    \end{remark}

    \subsection*{Calibration of the parameters} 
    To achieve the bound of Theorem~\ref{thm:squareloss}, SAEW is given the parameters $d_0$, $\alpha$, $U$, and $B$ beforehand. We provide here how to tune these parameters in order to sequentially get an estimator achieving high rate on its excess risk. To do so, we use a combination of well-known calibration techniques: doubling trick, meta-algorithm, and clipping.

    We only prove the calibration in the setting of linear regression with square loss (i.e., for Theorem~\ref{thm:squareloss} only and not for the general Theorem~\ref{thm:SAEW}). It remains an open question whether the calibration of the parameters can be performed in the general setting of Section~\ref{sec:known_parameters}. We leave this question for future research. Furthermore, for the sake of clarity the adaption to $Y$ (which is only necessary for clipping) is not considered here. However, it can be achieved simultaneously by updating the clipping range based on the past observations $Y_s$, $s\leq t-1$ (see~\cite[Section~4.5]{Gerchinovitz2011a}).

     The calibration algorithm (Algorithm~\ref{alg:calibration}) works as follows. We define large enough grids of parameters for each doubling session $j \geq 0$
    \begin{align}
        \cG_j = \Big\{(d_0, &\alpha,U,B) \in [1,\dots,d] \times \R_+^3 \quad \text{such that} \nonumber \\
          & d_0 \in \{0\} \cup \big\{2^k, k = 0,\dots, \lceil \log_2 d\rceil \big\} \nonumber \\
          & \alpha \in \big\{2^{k}, k = -2j + \lceil \log_2 (Bd_0/Y^2)\rceil, \dots, \nonumber \\
          & \pushright{j + \lceil \log_2 d_0\rceil \big\}} \nonumber \\
          & U \in \big\{2^k, k=-2j,\dots, 2j + \lceil 2 \log_2 Y\rceil  \big\} \nonumber \\
          & B \in \big\{2^k, k=-2j,\dots, 2j + \lceil 2 \log_2 Y\rceil \} \, \Big\} \label{eq:grid} \,.
    \end{align}
    For each set of parameters $p = (d_0,\alpha,U,B) \in \cG_j$, we perform a local version of SAEW to obtain an estimator $\tilde \theta_{p,j}$ at time $t=2^{j}-1$. Then, the calibration algorithm uses the online aggregation procedure BOA of \citet{Wintenberger2014} to make predictions from $t=2^j$ to $2^{j+1}-1$. Its predictions are based on online combinations of the (clipped) forecasts made by the $\tilde \theta_{p,j}$.   

    \begin{algorithm}[t!]
    \caption{Calibration algorithm}
    \label{alg:calibration}
    \begin{minipage}{0.45\textwidth}
    {\bfseries Parameters:} $Y>0$, $\delta>0$ \\[5pt]
    {\bfseries Initialization:} $t_0 = t = 1$ and $\bar \theta^{(0)} = 0$ \\[5pt]
    For each  $j = 0,1,\dots$ 
    \begin{itemize}
      \item Define the grid $\cG_j$ as in~\eqref{eq:grid}
      \item For parameters $p = (d_0,\alpha,U,B) \in \cG_j$:
      \begin{itemize}[leftmargin=10pt]
        \item Define $\delta_j = \delta/(2 (j+1)^2)$
        \item Run SAEW with parameter $(d_0,\alpha,U,B,\delta_j)$ for $t = 0,\dots, 2^{j}-1$ and get the estimator $\tilde \theta_{2^j-1}$, denoted by $\tilde \theta_{p,j}$.
        \item Define the clipped predictor 
            \[
                f_{p,j}:x \mapsto [x^\top \tilde \theta_{p,j}]_Y
            \]
            where $[\,\cdot\,]_Y := \max\big\{-Y,\min\{\,\cdot\,,Y\}\big\}$.
     \end{itemize}
      \item For $t = 2^j,\dots,2^{j+1}-1$, 
      \begin{itemize}[leftmargin=10pt]
        \item predict $\hat f_{t-1}(X_t)$ by performing BOA with experts $(f_{p,j})_{p \in \cG_j}$
        \item output the estimator $\tilde f_{t-1} = \bar f_{j}$
      \end{itemize}
      \item Define the average estimator \\[-5pt]
        \[
            \bar f_{j+1} = 2^{-j} \sum_{t=2^j}^{2^{j+1}-1} \hat f_{t-1}\,.
        \]
    \end{itemize}
    \end{minipage}
    \end{algorithm}

    \begin{theorem} \label{thm:calibration}
    Let $Y > \max_{t=1,\dots,T}|Y_t|$ almost surely. With probability $1-\delta$, the excess risk of the estimator $\tilde f_{T}$ produced by Algorithm~\ref{alg:calibration} is of order
      \begin{multline*}
          \cO_T \bigg(  \frac{Y^2}{T} \log \Big(\frac{(\log d)(\log T+\log Y)}{\delta}\Big) \\
         +  \frac{d_0 X^2 \sigma^2}{\alpha^\ast T } \log \Big(\frac{d\log T}{\delta}\Big)  \bigg) \,,
      \end{multline*}
      where $d_0 = \|\theta^\ast\|_0$ and $\alpha^\ast >0$ is the largest value of $\alpha$ satisfying Inequality~\eqref{ass:strong_convexity}.
    \end{theorem}
    The proof is postponed to Appendix~\ref{proof:calibration}. 

    \begin{remark} 
    Similarly to the restricting eigenvalue condition of the Lasso, we believe that the strong convexity condition for $\alpha^\ast$ might be necessary on subspaces of dimension lower than $d_0$ only. However, to do so, SAEW should be used with a subroutine that produces sparse $\hat \theta_{t-1}$. Up to our knowledge, such procedures do not exist for convex optimization in the $\ell_1$-ball. As stated previously, sparse procedures such as RDA of~\citet{Xiao2010} cannot be used as subroutines since they perform optimization in the $\ell_2$-ball and suffer a linear dependence on $d$. We leave this question for future work.
    \end{remark}
    \begin{remark}
    For the sake of clarity, the above result is only stated asymptotically. However the bound also holds in finite time up to universal multiplicative constant (as done in the proof). Additional negligible terms of order $\cO(1/T^2)$ then appear in the bound. Furthermore, the finite time bound also achieves the best of the two regimes (slow rate vs fast rate) as in Theorem~\ref{thm:squareloss}.
    \end{remark}

    \begin{remark}
    Theorem~\ref{thm:calibration} has been proven only for square linear regression. However, it also holds for any strongly-convex loss function, with locally bounded gradients (i.e., with LIST condition, see~\cite{Wintenberger2014}). 
    \end{remark}

    \begin{remark} 
    To perform the calibration, we left the original framework of Section~\ref{sec:prel}. First, because of the clipping, the estimators $\tilde f_{t-1}$ produced by Algorithm~\ref{alg:calibration} are not linear any-more. Second, the meta-algorithm implies that we can observe the gradients of all subroutines SAEW simultaneously. Tuning the parameters in the original setting is left for future work.
    \end{remark}

    \section{Simulations}\label{sec:as}

    In this section, we provide computational experiments on simulated data. We compare three online aggregation procedures:
    \begin{itemize}[topsep=0pt]
        \item RDA: a $\ell_1$-regularized dual averaging method as proposed by Algorithm~2 of~\citet{Xiao2010}. The method was shown to produce sparse estimators. It obtained good performance on the MNIST data set of handwritten digits \citep{LeCun1998}. We optimize the parameters $\gamma, \rho$, and $\lambda$ in hindsight on the grid~$\cE := \{10^{-5},\dots,10^3\}$.
        \item BOA: the Bernstein Online Aggregation of \citet{Wintenberger2014}. It proposes an adaptive calibration of its learning parameters and achieves the fast rate for the model selection problem (see~\cite{Nemirovski2000}). BOA is initially designed to perform aggregation in the simplex, for the setting of prediction with expert advice (see~\cite{Cesa-BianchiLugosi2006}). We use it together with the trick of~\citet{KivinenWarmuth1997} to extend it to optimization in the $\ell_1$-ball $\cB_1(0,\|\theta^\ast\|_1)$. 
        \item SAEW: the acceleration procedure as detailed in Algorithm~\ref{alg:SAEW}. We use BOA for the subroutines since it satisfies a regret bound of the form~\eqref{eq:reg_online_optimization}. For the parameters, we use $\delta = 0.95$, $U = \|\theta^\ast\|_1$ and $d_0 = \|\theta^\ast\|_0$. We calibrate $\alpha$ and $B$ on the grid $\cE$ in hindsight. 
    \end{itemize}
    Our objective here is only to show the potential of the acceleration of BOA for a well-chosen set of parameters in the general setting of Section~\ref{sec:known_parameters}. 

    \subsection{Application to square linear regression}
    We consider the square linear regression setting of Section~\ref{sec:squareloss}. We simulate  $X_t \sim \cN(0,1)$ for $d = 500$ and 
    \[
        Y_t = X_t^\top \theta^\ast + \epsilon_t \qquad \text{with} \quad  \epsilon_t \sim \cN(0,0.01) \quad \text{i.i.d.} \,,
    \]
    where $d_0 = \|\theta^\ast\|_0 = 5$, $\|\theta^\ast\|_1 = 1$ with non-zero coordinates independently sampled proportional to $\cN(0,1)$. 

    \begin{figure}[!ht]
       \begin{minipage}[t]{.52\linewidth}
        \centering\large 
        \includegraphics[height = 3cm]{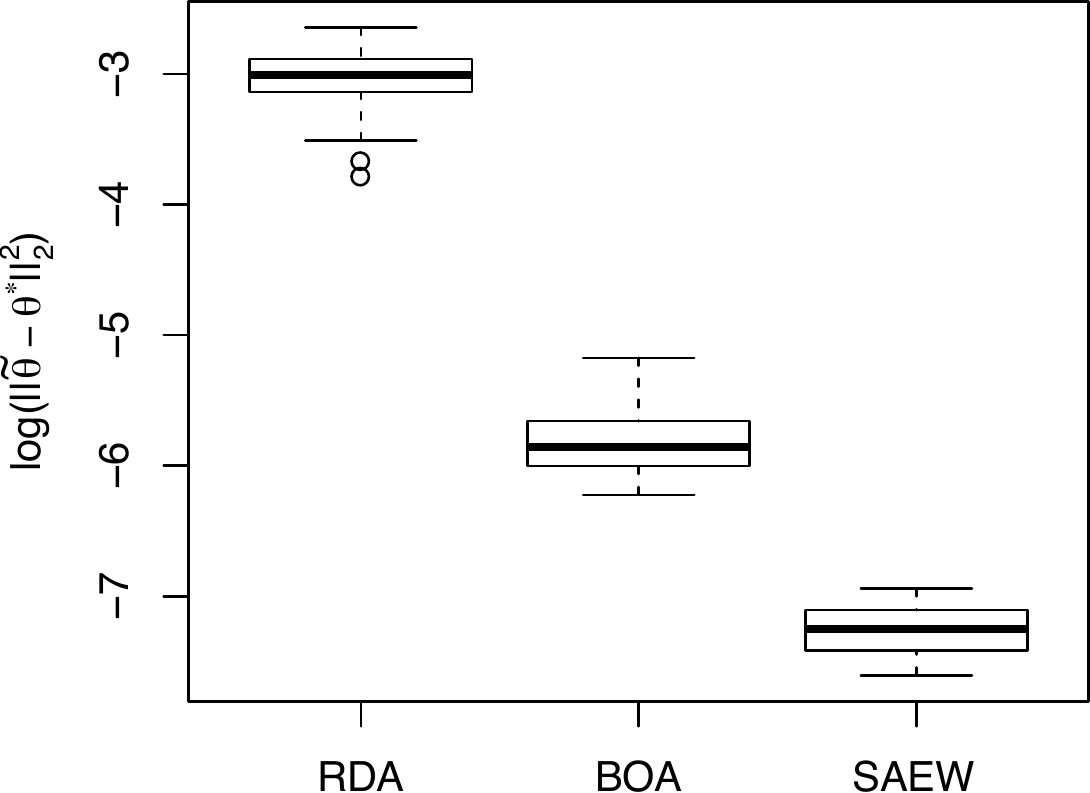}
        \subcaption{Square linear regression.} 
        \label{fig:boxplotL2}
        \end{minipage}
        \begin{minipage}[t]{.47\linewidth}
        \centering\large
        \includegraphics[height = 3cm]{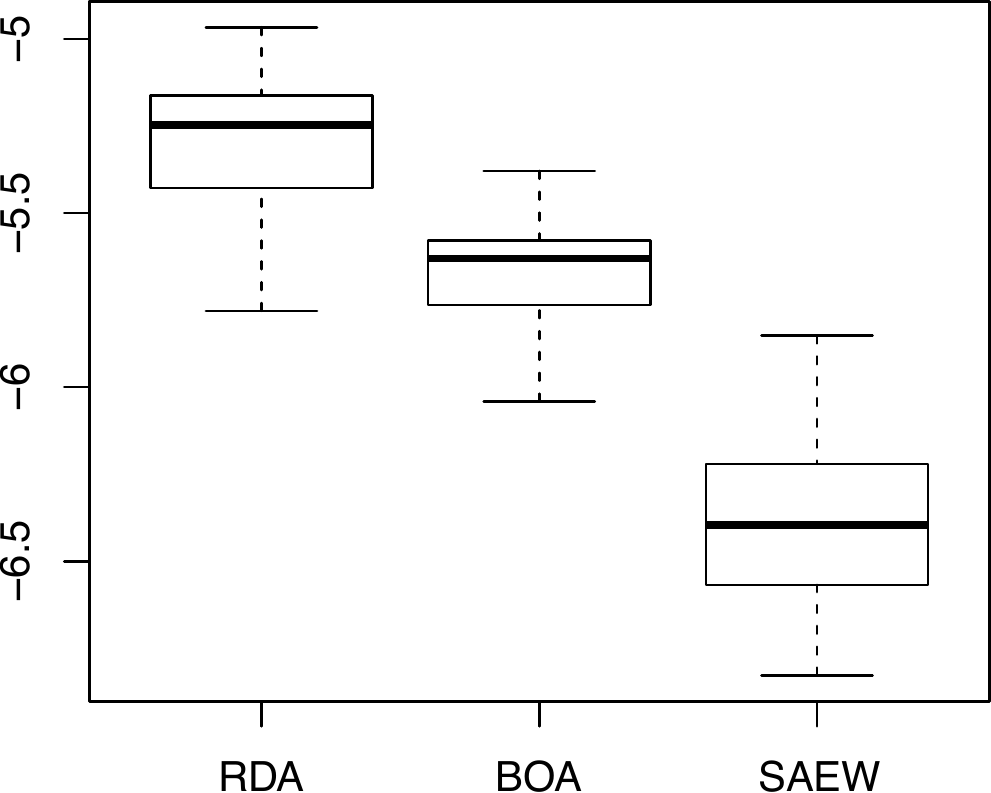}
        \subcaption{Quantile regression.}
        \label{fig:boxplotL2-quantile}
        \end{minipage}
        \caption{Boxplot of the logarithm of the $\ell_2$ errors of the estimators $\tilde \theta_T$ at time $T = 2\,000$ with $d = 500$.}
    \end{figure}

    Figure~\ref{fig:boxplotL2} illustrates the results obtained by the different procedures after the observation of $T = 2\,000$ data points. It plots the box-plot of the $\ell_2$ estimation errors of $\theta^\ast$, which is also approximatively the instantaneous risk, over 30 experiments. In contrast to BOA and SAEW, RDA does not have the knowledge of $\|\theta^\ast\|_1$ in advance. This might explain the better performance obtained by BOA and SAEW. Another likely explanation comes from the theoretical guarantees of RDA, which is only linear in $d$ (due to the sum of the squared gradients) though the $\ell_1$-penalization.
    
    In a batch context, the Lasso (together with cross-validation) may provide a better estimator for high dimensions $d$ (its averaged error would be $\log \tilde \theta_T \approx -8.8$ in Figure~\ref{fig:boxplotL2}). This is mostly due to two facts. First, because of the online setting, our online procedures are here allowed to pass only once through the data. If we allowed multiple passes, their performance would be much improved. Second, although BOA satisfies theoretical guarantees in $\sqrt{\log d}$, its performance is deeply deteriorated when $d$ becomes too large and does not converge before $T$ being very large. We believe our acceleration procedure should thus be used with sparse online sub-procedures instead of BOA, but we leave this for future research.

    \begin{figure}[!ht]
        \centering
        \includegraphics[width = 0.4\textwidth]{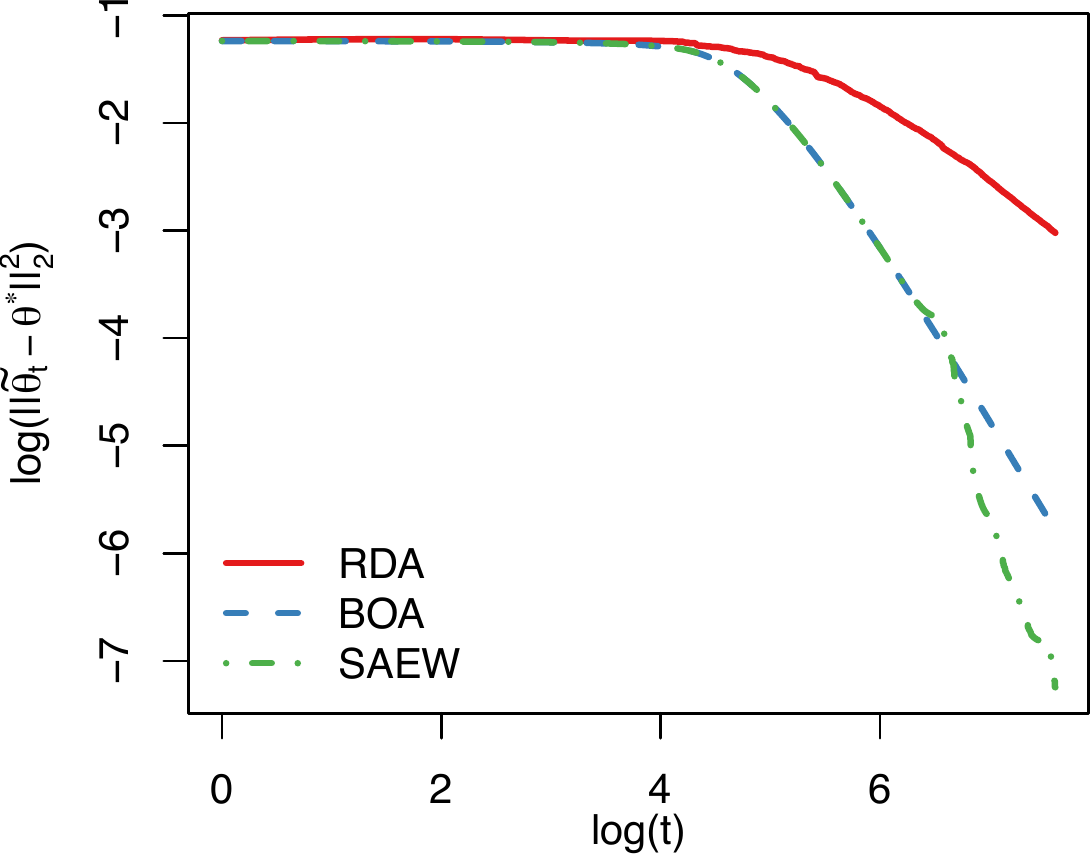}
        \caption{Averaged (over 30 experiments) evolution of the logarithm of the $\ell_2$ error.}\label{fig:evoL2}
    \end{figure}

    Figure~\ref{fig:evoL2} shows the decrease of the $\ell_2$-error over time in log/log scale. The performance is averaged over the 30 experiments. We see that SAEW starts by following BOA, until it considers to be accurate enough to accelerate the process (around $\log t \approx 6.2 $). Note that shortly after the acceleration start, the performance is shortly worse than the one of BOA. This can be explained by the doubling trick: the algorithm start learning again almost from scratch. The cumulative risks are displayed in Figure~\ref{fig:cumulative-risk-square}. SAEW and BOA  seem to achieve logarithmic cumulative risk, in contrast to RDA which seems to be of order $\cO(\sqrt{T})$. 
    
    \begin{figure}[!ht]
        \centering
        \includegraphics[width = .4\textwidth]{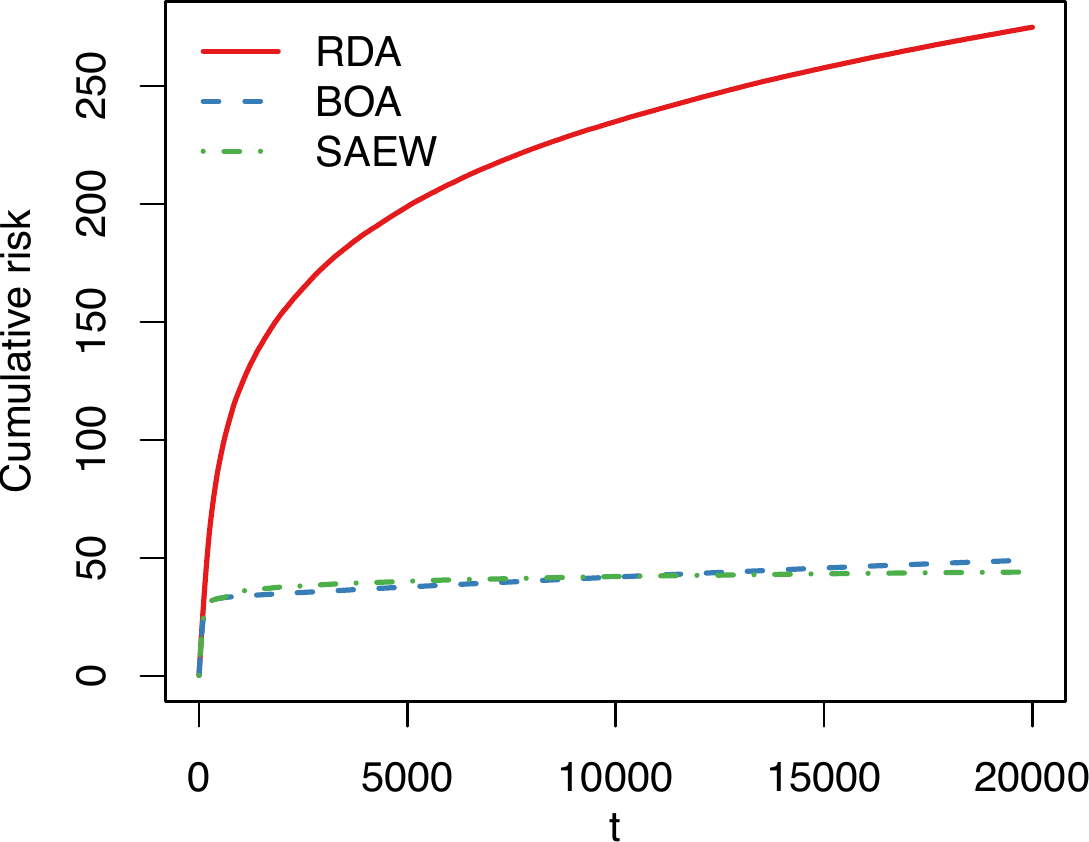}
        \caption{Averaged (over 30 runs) cumulative risk suffered by $\hat \theta_{t}$ for square linear regression.}\label{fig:cumulative-risk-square}
    \end{figure}

    In reality, the cumulative risk of BOA is of order $\cO(\sigma^2 \sqrt{T\log d} + \log d)$. In the previous experiment, because of the small value of the noise $\sigma^2 = 0.01$, the first term is negligible in comparison to the second one unless $T$ is very large. The behavior in $\sqrt{T}$ of BOA is thus better observed with higher noise and smaller dimension $d$, so that the first term becomes predominant. To illustrate this fact, we end the application on square linear regression with a simulation in small dimension $d_0 = d=2$ with higher noise $\sigma = 0.3$. Our acceleration procedure can still be useful to obtain fast rates. Figure~\ref{fig:cumulative-risk-square-smalld} shows that despite what seems on Figure~\ref{fig:cumulative-risk-square}, BOA does not achieve fast rate on its cumulative risk.

    \begin{figure}[!ht]
        \centering
        \includegraphics[width = .4\textwidth]{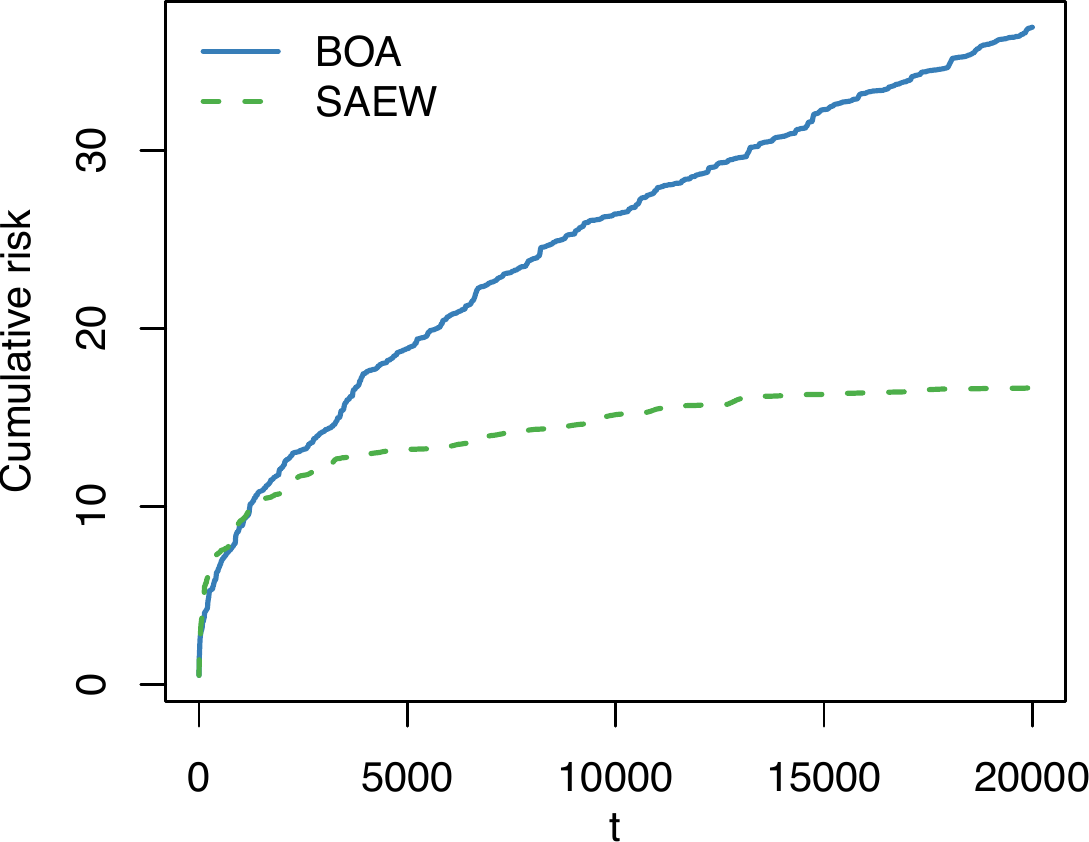}
        \caption{Cumulative risk suffered by $\hat \theta_{t}$ for square linear regression with $d=d_0=2$.}\label{fig:cumulative-risk-square-smalld}
    \end{figure}

    \subsection{Application to linear quantile regression}

    Let $\alpha \in (0,1)$. Here, we aim at estimating the conditional $\alpha$-quantile of $Y_t$ given $X_t$. A popular approach introduced by \citet{KoenkerBassett1978} consists in estimating the quantiles via the pinball loss defined for all $u \in \R$ by $\rho_\alpha(u) = u(\alpha - \mathds{1}_{u<0})$. It can be shown that the conditional quantile $q_\alpha(Y_t|X_t)$ is the solution of the minimization problem
    \[
        q_\alpha(Y_t|X_t) \in \argmin_{g} \E\big[\rho_\alpha\big(Y_t - g(X_t)\big) \big|X_t\big] \,.
    \]
    In linear quantile regression, we assume the conditional quantiles to be well-explained by linear functions of the covariates. \citet{SteinwartChristmann2011} proved that under some assumption the risk is strongly convex. We can thus apply our setting by using the loss functions 
    $
        \ell_t: \theta \mapsto \rho_\alpha\big(Y_t - X_t^\top \theta)
    $.

    We perform the same experiment as for linear regression $(Y_t,X_t)$, but we aim at predicting the $\alpha$-quantiles for $\alpha = 0.8$. To simulate an intercept necessary to predict the quantiles, we add a covariate 1 to the vector $X_t$. Figure~\ref{fig:boxplotL2-quantile} shows the improvements obtained by our accelerating procedure over the basic optimization algorithms. 

    In the next figures, to better display the dependence on $T$ of the procedures, we run them during a longer time $T = 10^5$ with $d = 100$ only.  

    \begin{figure}[!ht]
        \centering
        \includegraphics[width = 0.4\textwidth]{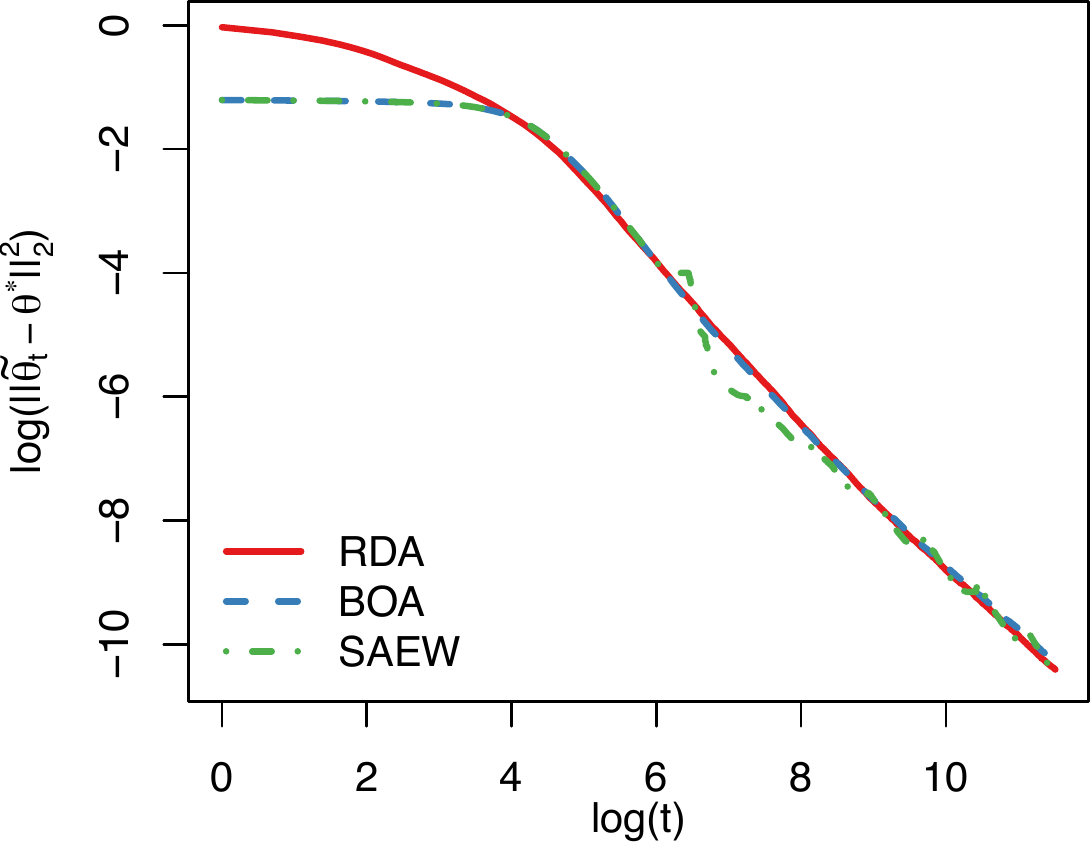}
        \caption{Averaged (over 30 runs) evolution of the logarithm of the $\ell_2$-error for quantile regression ($d=100$).}\label{fig:evoL2-quantile}
    \end{figure}

    Figure~\ref{fig:evoL2-quantile} depicts the decreasing of the $\ell_2$-errors of the different optimization methods (averaged over 30 runs). We see that unexpectedly most methods, although no theoretical properties, do achieve the fast rate $\cO(1/T)$ (which corresponds to a slope -1 on the log/log scale).  This explains why we do not really observe the acceleration on Figure~\ref{fig:evoL2-quantile}. However, we only show here the dependence on $t$ and not in $d$. 

    \begin{figure}[!ht]
        \centering
        \includegraphics[width = .45\textwidth]{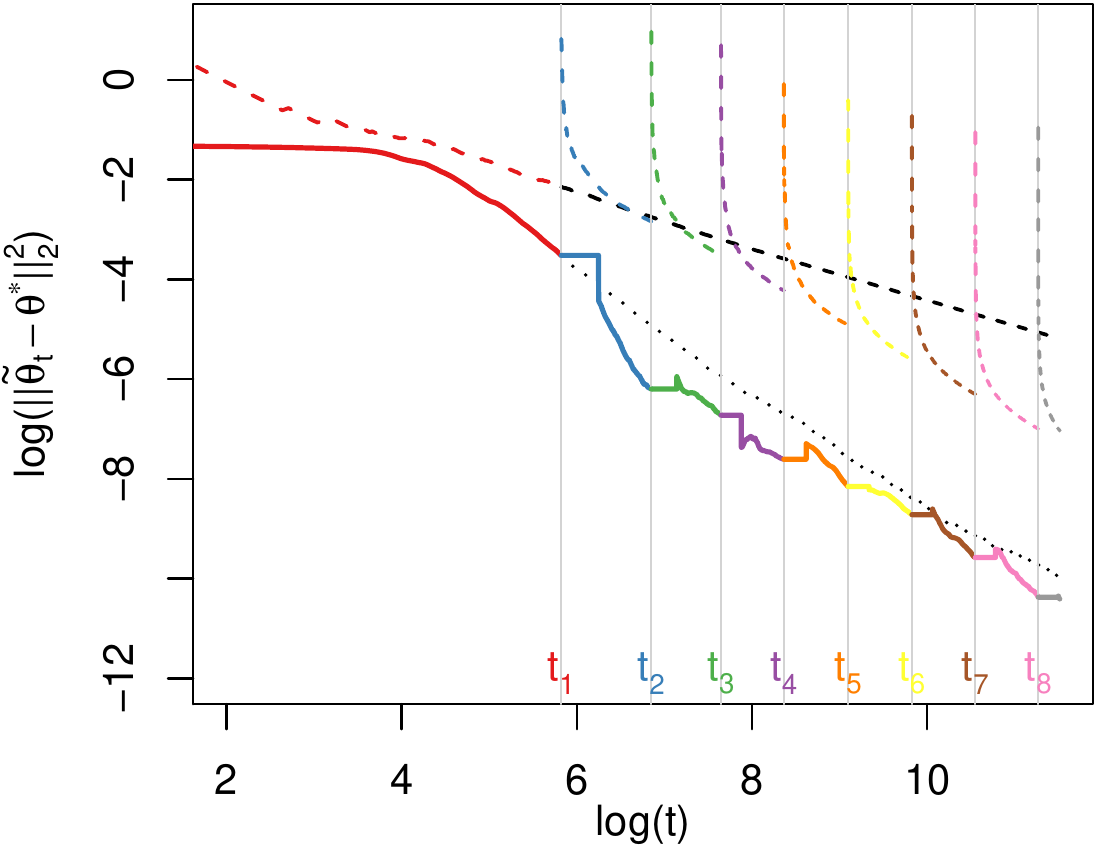}
        \caption{Logarithm of the $\ell_2$-norm of the averaged estimator $\tilde \theta_{t}$ during one run. The dashed lines represent the high probability $\ell_2$-bound estimated by SAEW on $\bar \theta_t$. The gray vertical lines are the stopping times $t_i$, $i\geq 1$. The first session is plotted in red, the second in blue,\dots The dotted and dashed black lines represent the performance (and the theoretical bound) that BOA would have obtained without acceleration.}\label{fig:acceleration-quantile}
    \end{figure}

    In Figure~\ref{fig:acceleration-quantile}, we show how the slow rate high-probability bound on BOA (slope $-1/2$ in log/log scale) is transformed by SAEW into a fast rate bound (slope -1).  To do so, it regularly restarts the algorithm to get smaller and smaller slow-rate bounds. Both BOA (dotted black line) and SAEW do achieve fast rate here though only SAEW guarantees it. It would be interesting in the future to prove the fast rate convergence for the averaged estimator produced by BOA in this context. The classical proof technique that uses a cumulative risk to risk conversion (with Jensen's inequality) will have however to be changed since the fast rate is not achieved for the cumulative risk (see Figure~\ref{fig:cumulative-risk-quantile}). 

    \begin{figure}[!ht]
        \centering
        \includegraphics[width = .4\textwidth]{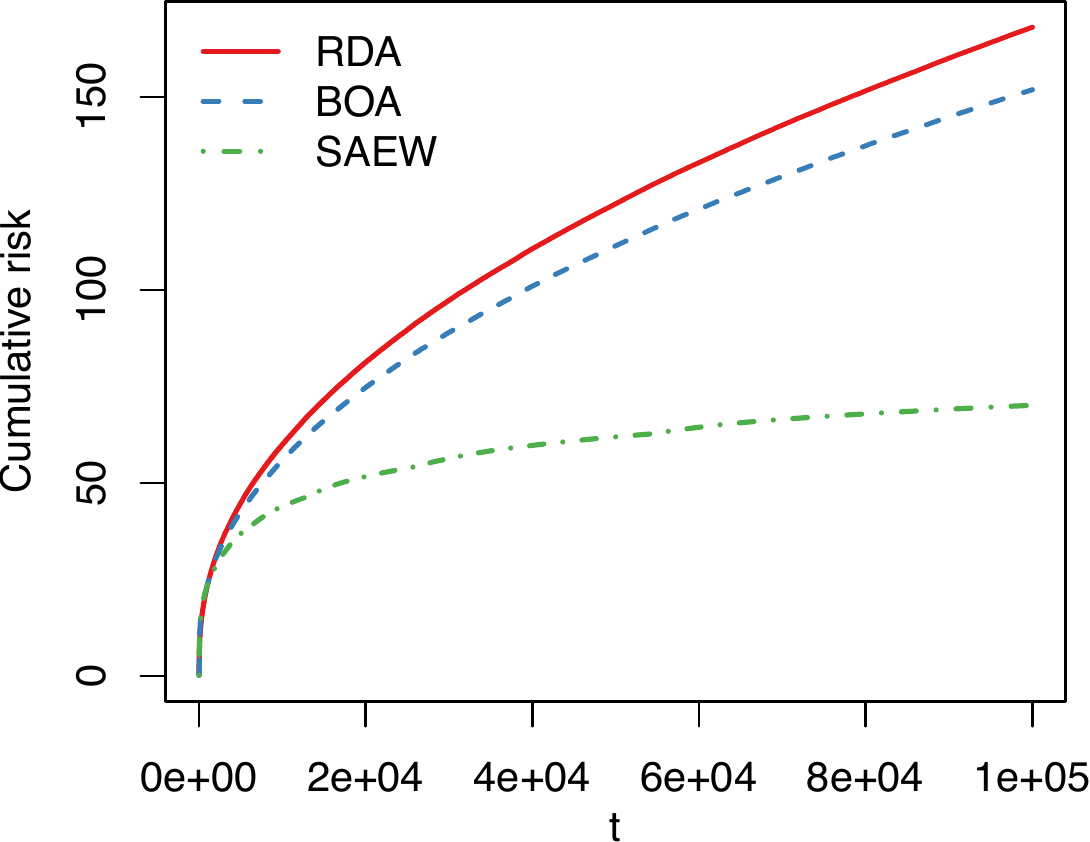}
        \caption{Averaged (over 30 runs) cumulative risk suffered by $\hat \theta_{t}$ for quantile regression ($d=100$).}\label{fig:cumulative-risk-quantile}
    \end{figure}

    \bibliographystyle{abbrvnat}
    \bibliography{biblio}
    \newpage
    \appendix
    {\Large \bfseries
\begin{center}
SUPPLEMENTARY MATERIAL
\end{center}}
  %%%%%%%%%%%%%%%%%%
% Preuves
%%%%%%%%%%%%%%%%%%

\section{Proofs}
\subsection{Lemma~\ref{lem:strong_convex}}
We first state Lemma~\ref{lem:strong_convex}, a classical result in strong convexity, as it will be useful in the proofs. 
It relates the $\ell_2$-error of an estimator with its excess risk when the risk is strongly convex.

\begin{lemma}
\label{lem:strong_convex}
If the risk is $2\alpha$-strongly convex, then \\[-.5em]
\[
     \| \theta -  \theta^\ast \|_2^2 \Big. \leq \alpha^{-1} \, \Risk(\theta)
\]
for all $\theta \in \R^d$.
\end{lemma}

\begin{proof}
\label{app:lemme1}
Let $\theta \in \R^d$, by \eqref{ass:strong_convexity} applied with $\theta_1 = \theta^\ast$ and $\theta_2 = \theta$, we get
\begin{multline*}
     \big\|\theta -\theta^\ast \big\|_2^2 \leq \alpha^{-1}  \E\big[\ell_t(\theta) - \ell_t(\theta^\ast)\big] \\
        + \alpha^{-1} \E\big[\nabla\ell_t(\theta^\ast)\big]^\top (\theta^\ast-\theta) \,.
\end{multline*}
But, $\E\big[\nabla\ell_t(\theta^\ast)\big]^\top (\theta^\ast-\theta) \leq 0$. Otherwise, taking into account the convexity of the domain, the direction $d =\theta - \theta^\ast$ is a decreasing feasible direction, which contradicts the optimality of $\theta^\ast$.
\end{proof}

\subsection{Proof of Theorem~\ref{thm:SAEW}}
\label{proof:SAEW}
Let $(\delta_i)$ be a non-increasing sequence in $(0,1)$ such that $\sum_{i=1}^\infty \delta_i \leq \delta$.

\paragraph{Step 1. Proof by induction that the subroutines always perform the optimization in the correct $\ell_1$-ball.}
We prove by induction on $i \geq 0$ that with probability at least $1 - \sum_{j=1}^i \delta_j$
\begin{equation}
    % \tag{$\mathcal{H}_i$}
    \label{eq:induction}
    \big\| \theta^\ast - [\bar \theta_{t_i-1}]_{d_0} \big\|_1 \leq U 2^{-i/2}  \,.
\end{equation}

$\mathcal{H}_0$ is satisfied by assumption since $\|\theta^\ast\|_1\leq U$ and $\smash{[\bar \theta_{t_0-1}]_{d_0} = [\bar \theta_0]_{d_0} = 0}$ (see SAEW for the definition of $[\bar \theta_0]$). 

Let $i \geq 0$ and assume \eqref{eq:induction}. The following Lemma (whose proof is postponed to Appendix~\ref{sub:proof_of_lemma_bounded_gradients}) states that the gradients are indeed upper-bounded by $B$ in sup-norm.
\begin{lemma}
    \label{lem:bounded_gradients}
    Let $i \geq 0$. Under \eqref{eq:induction}, for all $t \in [t_i,t_{i+1}-1]$, 
        $
            \big\|\nabla\ell_t(\hat \theta_{t-1})\|_\infty \leq B
        $
    almost surely.
\end{lemma}
Therefore, from the regret bound~\eqref{eq:reg_online_optimization}, the subroutine $\cS_i$ satisfies for all $t \in [t_i,t_{i+1}-1]$
\begin{multline*}
   \sum_{s=t_i}^{t} \ell_s(\hat \theta_{s-1}) -  \ell_s(\theta^\ast) \\
   \leq U 2^{-i/2} \bigg(a \sqrt{\sum_{s=t_i}^{t} \big\|\nabla \ell_s(\hat \theta_{s-1})\big\|_\infty^2} + bB\bigg) \,.
\end{multline*}
Bounding the cumulative risk with the regret thanks to Theorem~\ref{thm:riskregret} in Appendix~\ref{sec:riskregret}, it yields with probability at least $1-\sum_{j=1}^{i+1} \delta_{j}$,
\begin{equation}
    \sum_{s=t_i}^{t} \E[\ell_s](\hat \theta_{s-1}) -  \E[\ell_s](\theta^\ast) 
      \leq  U 2^{-i/2} \Err_t
    \label{eq:risk}
\end{equation}
where $\Err_t := a_i' \sqrt{\sum_{t_i}^{t} \big\|\nabla \ell_s(\hat \theta_{s-1})\big\|_\infty^2} + b_i'B$ with 
\begin{align}
    a_i' & := a + \sqrt{2} \sqrt{\log\Big(1 + \frac{1}{2} \log\Big(\frac{t-t_i+1}{2}\Big)\Big) - \log \delta_{i+1} }, \label{eq:defa}\\[-2em]
    \text{and} & \nonumber \\
    b_i' & := b  +  \frac{1}{2} + \log\Big(1 + \frac{1}{2} \log\Big(\frac{t-t_i+1}{2}\Big)\Big) - \log \delta_{i+1} \label{eq:defb} \,.
\end{align}
Thus, recalling that by definition (see SAEW)
\[
     \bar \theta_{t} := (t-t_i+1)^{-1} \sum_{s=t_i}^{t} \hat \theta_{s-1} \,,
\]
and because the losses are i.i.d., Jensen's inequality yields
\begin{eqnarray}
    \Risk(\bar \theta_t) \hspace*{-3ex}& = & \E [\ell_{t+1}](\bar \theta_t) - \E[\ell_{t+1}](\theta^\ast) \nonumber  \\
        & \stackrel{\text{Jensen}}{\leq} & (t-t_i+1)^{-1} \sum_{s=t_i}^{t} \E[\ell_s](\hat \theta_{s-1}) -  \E[\ell_s](\theta^\ast) \nonumber \\
        & \stackrel{\eqref{eq:risk}}{\leq} &  \frac{U \Err_t}{2^{i/2}(t-t_i+1)} \,. 
    \label{eq:riskbartheta}
\end{eqnarray}
Together with the strong convexity of the risk (Lemma~\ref{lem:strong_convex}), this entails
\begin{equation}
\big\|\bar \theta_t - \theta^\ast\big\|_2^2
\leq   \frac{U \Err_t}{\alpha 2^{i/2} (t-t_i+1)} \,. 
    \label{eq:taui}
\end{equation}
We thus control the $\ell_2$-error of $\bar \theta_t$. However, in order to control the $\ell_1$-error without paying a factor $d$, we need to truncate coordinates of $\bar \theta_t$. By definition of $[\bar \theta_{t}]_{d_0}$ (see SAEW), we have
\begin{equation}
    \label{eq:defthetacenter}
    [\bar \theta_{t}]_{d_0} \in \argmin_{\theta \in \R^d: \|\theta\|_0  \leq d_0} \big\{\big\|\bar \theta_t - \theta\big\|_2\big\} \,.
\end{equation}
Now, \eqref{eq:defthetacenter} together with $\|\theta^\ast\|_0 \leq d_0$ (by assumption) yields
\begin{equation}
    \label{eq:thetaithetacenter}
    \big\|\bar \theta_t - [\bar \theta_{t}]_{d_0}  \big\|_2 \leq \big\|\bar \theta_t - \theta^\ast \big\|_2 \,.
\end{equation}
Furthermore, because both $\|\theta^\ast\|_0 \leq d_0$ and $\big\|[\bar \theta_{t}]_{d_0}\|_0 \leq d_0$, we have
\begin{equation}
    \label{eq:thetacenterthetastar}
    \|[\bar \theta_{t}]_{d_0} - \theta^\ast\|_0 \leq 2d_0 \,.
\end{equation}
Therefore, with probability at least $1-\sum_{j=0}^{i+1} \delta_j$
\begin{align}
    \big\| [\bar \theta_{t}]_{d_0} & - \theta^\ast \big\|_1  
         \stackrel{\eqref{eq:thetacenterthetastar}}{\leq}  \sqrt{2d_0}\, \big\| [\bar \theta_{t}]_{d_0} - \theta^\ast \big\|_2  \nonumber \\
        & \leq  \sqrt{2d_0} \Big( \big\| [\bar \theta_{t}]_{d_0} - \bar \theta_t\big\|_2   + \big\| \bar \theta_{t} - \theta^\ast \big\|_2\Big) \nonumber\\
        & \stackrel{\eqref{eq:thetaithetacenter}}{\leq} 2\sqrt{2d_0}  \big\| \bar \theta_t - \theta^\ast \big\|_2 \nonumber \\
        & \stackrel{\eqref{eq:taui}}{\leq} 2\sqrt{2d_0  \alpha^{-1} U \Err_t 2^{-i/2}(t-t_i+1)^{-1}}  \nonumber \\
        & =:  \epsilon_{t}\,, 
    \label{eq:defepsilont}
\end{align}
where the last equality holds by definition of $\epsilon_t$ (see SAEW). Finally, $(\mathcal{H}_{i+1})$ is fulfilled by definition of $t_{i+1}$ (see SAEW), which satisfies
$
  \epsilon_{t_{i+1}-1} \leq U 2^{-(i+1)/2} 
$.
The induction is thus completed.

In the rest of the proof, we consider that \eqref{eq:induction} are satisfied for all $i\geq 0$. This occurs with probability $1-\sum_{j=1}^\infty \delta_j \geq 1- \delta$ as stated by Step~1. 

\paragraph{Step 2. Fast rate for the excess risk of $\tilde \theta_t$.}  
First, we prove that the excess risk of $\tilde \theta_t$ is upper-bounded as
\begin{equation}
    \label{eq:fastrate}
    \Risk(\tilde \theta_t) \leq \frac{d_0 B^2}{\alpha} \left(\frac{2^7 a'^2}{t} + \frac{2^{11}b'^2}{t^2}\right) + \frac{2\alpha U^2}{d_0 t^2}  \,,
\end{equation}
for all $t\geq 1$, where $a'= a'_{\lfloor 2\log_2t\rfloor}$ and $b'= b'_{\lfloor 2\log_2 t\rfloor}$.

To do so, we start from the risk inequality~\eqref{eq:riskbartheta}. From the definition of $\epsilon_t$ (see~\eqref{eq:defepsilont}), we get 
\begin{equation}
    \label{eq:riskthetaepsilon}
      \Risk(\bar \theta_t) \leq \frac{\alpha \epsilon_t^2}{8d_0}\,, \qquad t \geq 1\,.
\end{equation}
Thus by definition of $\tilde \theta_t := \bar \theta_{\argmin_{s\leq t} \epsilon_s}$, we have
\begin{equation}
      \Risk(\tilde \theta_t) 
         \leq \frac{\alpha \min_{s\leq t} \epsilon_s^2}{8d_0} 
    \label{eq:riskthetatilde}
\end{equation}
We conclude the proof with the following lemma proved in Appendix~\ref{sub:epsilont}
\begin{lemma} 
\label{lem:epsilont}
Let $i \geq 0$. Let $t_i -1 \leq t \leq t_{i+1}$, then
\[
    \min_{s\leq t} \epsilon_s \leq U\left( \frac{\sqrt{2}\gamma a_{i}'}{\sqrt{t}} + \frac{2+ 4\gamma b_{i}'}{t}\right) \,,
\]
where $\gamma := 2^4 d_0 B/(\alpha U)$.
\end{lemma}

Let $i\geq 0$ such that $t_i-1\leq t\leq t_{i+1}$. Lemma~\ref{lem:epsilont} together with \eqref{eq:riskthetatilde} and $(x+y)^2 \leq 2x^2+2y^2$ for $x,y\geq0$, yields
\begin{align}
    \Risk(\tilde \theta_t) 
        & \leq \frac{\alpha U^2\gamma^2}{8d_0} \bigg(\frac{\sqrt{2} a_{i}'}{\sqrt{t}} + \frac{2\gamma^{-1} + 4b_{i}'}{t}\bigg)^2 \label{eq:riskthetatilecarre}\\
        & \leq \frac{ \alpha U^2\gamma^2}{d_0} \left(\frac{a_i'^2}{2t} +  \frac{2\gamma^{-2}+8b_i'^2}{t^2}\right)  \,.
    \label{eq:risktildetehetatend}
\end{align}
Now, remark that if $i \geq 2 \log t$, then $\epsilon_{t_i-1} \leq U 2^{-i} \leq U/t$ and from \eqref{eq:riskthetatilde}, $\Risk(\tilde \theta_t) \leq \alpha U^2/ (8d_0 t^2)$. Together, with \eqref{eq:risktildetehetatend}, we get
\begin{equation*}
    \Risk(\tilde \theta_t) \leq \frac{\alpha U^2 \gamma^2}{d_0} \left(\frac{ a'^2}{2t} + \frac{2\gamma^{-2}+ 8b'^2}{t^2}\right)  \,,
\end{equation*}
with $a'= a'_{\lfloor 2\log_2 t\rfloor}$ and $b'= b'_{\lfloor 2\log_2 t\rfloor}$.  Substituting $\gamma = 2^4 d_0 B/(\alpha U)$ concludes the proof of Inequality~\eqref{eq:fastrate}.

\paragraph{Step 3. Slow rate for the excess risk of $\tilde \theta_t$.}
Now, we prove that
\begin{equation}
    \label{eq:slowrate}
   \Risk(\tilde \theta_t) \leq  UB \left(\frac{a'}{\sqrt{t/2}} + \frac{4b'}{t}\right) + \frac{\alpha U^2}{8d_0 t},\ t\geq 1 \,.
\end{equation}

For small values of $t$, the slow rate will be satisfied  from the initial bound of the subroutine during the first session. At some time $\tau >0$, the fast rate becomes better than the slow rate. This splitting time is defined as the solution of the equality
\begin{equation}
    \label{eq:deftau}
    \frac{\Err_{t_1-1}}{t_1-1} = B\Big(\frac{\sqrt{2}a'}{\sqrt{\tau}}+\frac{2 \gamma^{-1}+ 4 b'}{\tau}\Big)\,.
\end{equation}

Let $t \geq 1$. To control $\Risk(\tilde \theta_t)$, we distinguish three cases:
\begin{itemize}
    \item if $t \leq t_1-1$, then, since by definition of $\epsilon_s$
    \[
        \argmin_{s\leq t} \frac{\Err_s}{s} = \argmin_{s\leq t} \varepsilon_s \,,
    \]
    we get from Inequality~\eqref{eq:riskbartheta} that
    \begin{align*}
        \Risk(\tilde \theta_t) 
            & = \Risk(\bar \theta_{\argmin_{s\leq t} \epsilon_s}) \\
            & \leq  U 2^{-0/2} \min_{s\leq t} \frac{\Err_s}{s} \\
            & \leq U \frac{\Err_t}{t} \,.
    \end{align*}
    By definition of $\Err_t$ (see~\eqref{eq:risk}) and upper-bounding the gradients by $B$, we get
    \[
        \Risk(\tilde \theta_t) \leq UB \Big( \frac{a_0'}{\sqrt{t}} + \frac{b_0'}{t}\Big)\,.
    \]
    \item if $t_1 \leq t \leq \tau$, then following the same reasoning as above, we have
    \[
        \Risk(\tilde \theta_t) \leq U \frac{\Err_{t_1-1}}{t_1-1} \,,
    \]
    which yields by definition of $\tau$ (see Equality~\eqref{eq:deftau}) and by using $t \leq \tau$:
    \begin{align*}
        \Risk(\tilde \theta_t) 
            & \leq UB \Big( \frac{\sqrt{2}a'}{\sqrt{\tau}} + \frac{2\gamma^{-1}+4b'}{\tau}\Big)\\
            & \leq UB \Big( \frac{\sqrt{2}a'}{\sqrt{t}} + \frac{2\gamma^{-1}+4b'}{t}\Big)\,.
    \end{align*}

    \item if $\tau \leq t$, since by definition of $t_1$ (see SAEW), $\epsilon_{t_1-1} \leq U/2$, then by definition of $\epsilon_{t_1-1}$ (see~\eqref{eq:defepsilont}),
    \[
        2\sqrt{2d_0  \alpha^{-1} U \frac{\Err_{t_1-1}}{t_1-1}} \leq \frac{U}{2}\,,
    \]
    and thus taking the square and rearranging the terms
    \[
        \frac{d_0}{\alpha} \leq \frac{U}{2^5} \Big(\frac{t_1-1}{\Err_{t_1-1}}\Big)\,.
    \]
    Using the definition of $\gamma = 2^4 d_0 B /(\alpha U)$ and substituting $\Err_{t_1-1}$ with Equality~\eqref{eq:deftau}, this yields
    \[
        \frac{\alpha U^2\gamma^2}{8d_0} = \frac{2^5 d_0 B^2}{\alpha} \leq U B \Big(\frac{\sqrt{2}a'}{\sqrt{\tau}}+\frac{2\gamma^{-1}+4b'}{\tau}\Big)^{-1} \,.
    \]
    Finally from Inequality~\eqref{eq:riskthetatilecarre}, and using $\tau \leq t$
    \[
        \Risk(\tilde \theta_t) \leq UB \left(\frac{\sqrt{2}a'}{\sqrt{t}} + \frac{2\gamma^{-1} + 4b'}{t}\right)\,.
    \]
\end{itemize}

Combining the three cases together and substituting $\gamma = 2^4 d_0 B/(\alpha U)$, concludes the proof of Inequality~\eqref{eq:slowrate}.

\paragraph{Step~4. Conclusion of the proof}
Combining Inequalities~\eqref{eq:fastrate} and~\eqref{eq:slowrate}, we get the risk inequality stated in the theorem for $\tilde \theta_t$. 
It only remains to choose $\delta_j = \delta / (j+1)^2$ so that $\sum_{j=1}^\infty \delta_j\leq \delta$ and to control $a'= a'_{\lfloor 2\log_2 t\rfloor}$ and $b'= b'_{\lfloor 2\log_2 t\rfloor}$. From~\eqref{eq:defa}, we can use $\delta_{\lfloor 2\log_2 t\rfloor+1} \geq \delta / (1+2\log_2 t)^2$ and $T_i \leq t$. Straighforward calculation yields that  $a'- a$ is lower than
\begin{align*} 
    & \sqrt{2\big(\log(1 + 1/2 \log (t/2)) - \log \delta + 2\log(1+ 2 \log_2 t)\big)} \\
    & \leq \sqrt{6 \log(1 + 3 \log t) - 2 \log \delta}.
\end{align*}
Similarly, for $b'-b$. It is upper-bounded by
\begin{align*}
    & \frac{1}{2} + \log\big(1 + (1/2)\log(t/2)\big) - \log \delta + 2\log(1+ 2 \log_2 t) \\
    & \leq 1/2 + 3 \log(1+3\log t) - \log \delta\,.
\end{align*}
This concludes the proof.

\subsection{Proof of Lemma~\ref{lem:bounded_gradients}} % (fold)
\label{sub:proof_of_lemma_bounded_gradients}

Since by assumption $B \geq \max_{\theta:\|\theta\|_1\leq 2U} \|\nabla \ell_t(\theta)\|_\infty$ a.s. Therefore, it suffices to show that $\|\hat \theta_{t-1}\|_1 \leq 2U$. By definition of the session $\cS_i$, 
    \[
        \hat \theta_{t-1} \in \cB_1([\bar \theta_{t_i-1}]_{d_0},U2^{-i/2}) \,.
    \]
Thus:
\begin{itemize}[topsep=-3pt,parsep=2pt,itemsep=2pt]
    \item if $i=0$, since $[\bar \theta_{0}]_{d_0} = 0$, $\|\hat \theta_{t-1}\|_1\leq U$.
    \item if $i = 1$, then since $\|[\bar \theta_{t_1-1}]\|_1\leq U$ as a truncated average of vectors in $\cB_1(0,U)$, we have
\begin{align*}
    \|\hat \theta_{t-1}\|_1 
        & \leq \|\hat \theta_{t-1} - [\bar \theta_{t_1-1}]_{d_0}\|_1 + \|[\bar \theta_{t_1-1}]_{d_0}\|_1 \\
        & \leq U/\sqrt{2} + U \leq 2 U;
\end{align*}
    \item otherwise, $i \geq 2$ and $\|\hat \theta_{t-1}\|_1$ is bounded by
    \begin{multline*} 
        \|\hat \theta_{t-1} - [\bar \theta_{t_i-1}]_{d_0}\|_1 + \|[\bar \theta_{t_i-1}]_{d_0} - \theta^\ast\|_1 + \|\theta^\ast\|_1 \\
        \stackrel{\eqref{eq:induction}}{\leq} U2^{-i/2} + U2^{-i/2} + U \leq 2U \,.
    \end{multline*}
\end{itemize}
Putting the tree cases together, $\|\hat \theta_{t-1}\|_1\leq 2U$, which concludes the proof.

\subsection{Proof of Lemma~\ref{lem:epsilont}}
\label{sub:epsilont}

It is enough to control $\epsilon_{t_i-1} \geq \min_{s\leq t} \epsilon_s$. To do so, we prove that for every $j \geq 0$, $T_j := t_{j+1} - t_j$ cannot be too large, so that at time $t$, $i$ will be at least of order $\log_2 t$.

Let $j \geq 0$. We can assume $t_{j+1}>t_j$, otherwise $T_j = 0$. 
Thus, from the bound on the gradients (Lemma~\ref{lem:bounded_gradients}) and from the definition of $\Err_t$ (see~\eqref{eq:risk}) for all $t \in [t_j+1,t_{j+1}]$,
  \begin{equation}
    \label{eq:errt}
    \Err_{t-1} \leq B (a_j'  \sqrt{t-t_j} + b_j') \,,
  \end{equation}
and from the definition of $\epsilon_{t-1}$ (see ~\eqref{eq:defepsilont})
  \[
    \epsilon_{t-1} \leq 2\sqrt{2d_0  \alpha^{-1} U B \frac{a_j'\sqrt{t-t_j}  + b_j'}{2^{j/2}(t-t_j)}} \,.
  \]
Since by definition, $t_{j+1}$ is the smallest integer after $t_j$ that satisfies $\epsilon_{t_{j+1}-1} \leq U2^{-(j+1)/2}$, we have $\epsilon_{t_{j+1}-2} \geq U 2^{-(j+1)/2}$. This implies
\begin{align*}
   & 2  \sqrt{2d_0  \alpha^{-1} U B \frac{a_j' \sqrt{T_j-1}+ b_j'}{2^{j/2}(T_j-1)} } \geq U 2^{-(j+1)/2} \\
    & \pushright{\Leftrightarrow  \,  2^{j/2} \underbrace{2^4 d_0  \alpha^{-1} U^{-1}B}_{:= \gamma} \big(a_j' \sqrt{T_j-1} +  b_j' \big)  \geq T_j-1}
\end{align*}
Then, by solving a second order equation 
in $\sqrt{T_j-1}$ (see for instance~\cite[Lemma~10]{GaillardStoltzEtAl2014}), the above inequality entails
\begin{equation}
   T_j   \leq  1 + 2^{j} \, \gamma^2 a_j'^2 + 2^{j/2} \gamma b_j' \,. \label{eq:Ti}
\end{equation}
Therefore, summing over $j=0,\dots,i$ 
\begin{align*}
  t_{i+1} & = \cancel{t_0} + \sum_{j=0}^{i} T_j  \\
    & \leq \sum_{j=0}^{i} \big(1 + 2^{j} \, \gamma^2 a_j'^2 + 2^{j/2} \gamma b_j \big)\\
    & \leq 2^{1+i} \, \gamma^2 a_{i}'^2 + (1+\sqrt{2})2^{(i+1)/2} \gamma b_{i}' + i+1 \\
    & \leq 2^{1+i} \, \gamma^2 a_{i}'^2 + 2^{(i+1)/2} \sqrt{2} \big(2\gamma b_{i}' + 1\big)  \,,
\end{align*}
where the last inequality is because $2^{(i+1)/2}\geq \sqrt{2}(i+1)$ for $i\geq 0$.
Solving the second-order inequality in $2^{(i+1)/2}$ we get
\[
    2^{-(i+1)/2} \leq \frac{\gamma a_{i}'}{\sqrt{t_{i+1}}} + \sqrt{2}\frac{1+ 2\gamma b_{i}'}{t_{i+1}} \,.
\]
Thus, since $\epsilon_{t_i-1} \leq U 2^{-i/2}$, we have
\[
    \epsilon_{t_i-1} \leq U \gamma \left(\frac{\sqrt{2} a_{i}'}{\sqrt{t_{i+1}}} + \frac{2\gamma^{-1} + 4b_{i}'}{t_{i+1}}\right) \,.
\]
The proof of Lemma~\ref{lem:epsilont} finally follows using $t\leq t_{i+1}$.

\subsection{Proof of Theorem~\ref{thm:cumulativerisk}}
With probability $1-\delta$, all inequalities provided in the proof of Theorem~\ref{thm:SAEW} are satisfied. We also consider the notation of the previous proof. Let $t \geq 1$.

\paragraph{Step~1. Slow rate} 
We remark that for any $i\geq 0$,
\begin{align}
\sum_{s=t_i}^{(t_{i+1}-1)\wedge t} &  \E[\ell_s](\hat \theta_{s-1}) - \E[\ell_s](\theta^\ast) 
     \stackrel{\eqref{eq:risk}}{\leq} U 2^{-i/2} \Err_{(t_{i+1}-1)\wedge t} \nonumber\\
    & \stackrel{\eqref{eq:errt}}{\leq} UB2^{-i/2} (a_i' \sqrt{t} + b_i') \label{eq:slowratesmsalli}
\end{align}
where, in the last inequality, we use that $(t_{i+1}-1)\wedge t \leq t$ and $t_i \geq 1$. We will use this inequality for $i \leq \lfloor 2\log t\rfloor$. For $i > \lfloor 2\log t\rfloor$, we use the fact that the gradients are bounded by $B$, so that by convexity of the risk
\begin{align}
\sum_{s=t_i}^{(t_{i+1}-1)\wedge t} & \E[\ell_s](\hat \theta_{s-1}) - \E[\ell_s](\theta^\ast) \nonumber \\
&  \leq \sum_{s=t_i}^{(t_{i+1}-1)\wedge t} \|\E[\nabla\ell_s](\hat \theta_{s-1})\big\|_\infty \|\hat \theta_{s-1}-\theta^\ast\|_1  \nonumber \\
& \leq U B 2^{-i/2} t \,. \label{eq:sumlargei}
\end{align}
Summing~\eqref{eq:slowratesmsalli} over $i=0,\dots,\lfloor 2 \log_2 t\rfloor$ and~\eqref{eq:sumlargei} over $i = \lceil 2\log_2 t \rceil,\dots,\infty$, we get
\begin{align}
  \Risk_{1:t}(\hat \theta_{0:(t-1)}) 
    &:= \sum_{s=1}^{t} \E[\ell_s](\hat \theta_{s-1}) - \E[\ell_s](\theta^\ast) \nonumber \\
    & \leq  UB \sum_{i=0}^{\lfloor2\log_2 t\rfloor} 2^{-i/2} (a_i' \sqrt{t} + b_i') \nonumber  \\
    & \pushright{+ UBt \sum_{i = \lceil 2\log_2 t\rceil}^\infty 2^{-i/2}} \,. \label{eq:2sumsslow}
\end{align}
The second sum is controlled as
\[
    \sum_{i = \lceil 2 \log_2 t \rceil}^\infty  2^{-i/2}  \leq t^{-1} \sum_{i = 0}^\infty  2^{-i/2} \,.
\]
Thus, since $\sum_{i=0}^\infty 2^{-i/2} = 2+\sqrt{2}\leq 4$, we have
\[
    \Risk_{1:t}(\hat \theta_{0:(t-1)}) \leq 4 UB (a' \sqrt{t} + b') + 4 UB \,,
\]
where we recall that $a'= a'_{\lfloor 2\log_2 t\rfloor}$ and $b'= b'_{\lfloor 2\log_2 t\rfloor}$.
This concludes Step~1.

\paragraph{Step~2. Fast rate} Let us now prove the fast rate
\begin{multline*}
    \Risk_{1:t}\big(\hat \theta_{0:(t-1)}\big) \leq \frac{ 2^5 d_0 B^2}{\alpha} a'^2 \log_2 t \\
        + 4 B U (1+b') + U^2 \frac{\alpha}{8 d_0} \,,
\end{multline*}
for all $t \geq 1$.

First, we remark that similarly to~\eqref{eq:riskthetaepsilon}, we get for all $i\geq 0$ that
\begin{align}
\sum_{s=t_i}^{t_{i+1}-1} \E[\ell_s](\hat \theta_{s-1}) - \E[\ell_s](\theta^\ast)  & \stackrel{\eqref{eq:risk}}{\leq} U \frac{\Err_{t_{i+1}-1}}{2^{-i/2} T_i} T_i \nonumber \\
        & \stackrel{\eqref{eq:defepsilont}}{\leq} \frac{\alpha \varepsilon_{t_{i+1}-1}^2}{8d_0} T_i \nonumber \\
        & \leq \frac{\alpha U^2 2^{-i}}{16d_0} T_i
        \label{eq:sumsmalli}
\end{align}
where the last inequality is because $\varepsilon_{t_{i+1}-1} \leq U2^{-(i+1)/2}$ by definition of $t_{i+1}$ (see SAEW). We will use this inequality for $i \leq \lfloor 2\log t\rfloor$. 
Summing~\eqref{eq:sumsmalli} over $i=0,\dots,\lfloor 2 \log_2 t\rfloor$ and~\eqref{eq:sumlargei} over $i = \lceil 2\log_2 t \rceil,\dots,\infty$, we get
\begin{align}
  \Risk_{1:t}(\hat \theta_{0:(t-1)}) 
    &:= \sum_{s=1}^{t} \E[\ell_s](\hat \theta_{s-1}) - \E[\ell_s](\theta^\ast) \nonumber \\
    & \leq  \frac{U^2\alpha}{2^4d_0} \sum_{i=0}^{\lfloor2\log_2 t\rfloor}  2^{-i} T_i  \nonumber \\
    & \pushright{+ UBt \sum_{i = \lceil 2\log_2 t\rceil}^\infty 2^{-i/2}} \,. \label{eq:2sums}
\end{align}
We upper bound both sums. The second one is controlled as we did for~\eqref{eq:2sumsslow}. The first one is upper-bounded thanks to \eqref{eq:Ti}
\begin{align*}
    & \sum_{i=0}^{\lfloor 2 \log_2 t\rfloor}  2^{-i} T_i 
       \leq \sum_{i=0}^{\lfloor 2 \log_2 t\rfloor}  \Big( \gamma^2 a_i'^2  + 2^{-i/2} \gamma b_i' + 2^{-i} \Big)\\
      &\pushright{ \leq 2 \gamma^2 a'^2 \log_2 t + 4\gamma b' + 2 \,.}
\end{align*}
 Therefore, substituting the two sums into~\eqref{eq:2sums}, the cumulative risk $\Risk_{1:t}(\hat \theta_{0:(t-1)})$ is upper-bounded by
\[
    \frac{U^2\alpha}{2^4 d_0} \Big( 2 \gamma^2 a'^2 \log_2 t + 4\gamma b'+2\Big) + 4 UB
     \,,
\]
which, by substituting $\gamma = 2^4 d_0 B/(\alpha U)$, is equal to
\[
    \frac{ 2^5 d_0 B^2}{\alpha} a'^2 \log_2 t + 4 B U (1+b') + \frac{\alpha U^2}{8 d_0} \,.
\]
This concludes the proof.

\subsection{Proof of Theorem~\ref{thm:squareloss}}
\label{proof:squareloss}

Let first check that we are indeed in the setting of Theorem~\ref{thm:SAEW}. The risk is strongly convex because for any $\theta_1,\theta_2\in \R^d$
    \begin{align*}
        &\E[\ell_t(\theta_1) - \ell_t(\theta_2)]  
              = \E \Big[ (Y_t - X_t^\top \theta_1)^2 - (Y_t - X_t^\top \theta_2)^2\Big] \\
             & = \E\Big[  -2(Y_t - X_t^\top \theta_1) X_t^\top(\theta_1 - \theta_2) - \big(X_t^\top(\theta_1 - \theta_2)\big)^2\Big] \\
             & = \nabla \E[\ell_t](\theta_1)^\top (\theta_1 - \theta_2) - (\theta_1 -\theta_2)^\top \E\big[X_tX_t^\top\big] (\theta_1 -\theta_2) \,.
    \end{align*}
Assumption~(\ref{ass:strong_convexity}) is thus satisfied with $\alpha = \lambda_{\min}(\E\big[X_tX_t^\top\big])$. Besides, for all $\theta$ such that $\|\theta\|_1 \leq 2U$, we have
\[
    \|\nabla \ell_t(\theta)\|_\infty = \|2(Y_t-X_t^\top \theta)X_t\|_\infty \leq 2(Y + 2XU)X = B \,.
\]

Now, we mimic the proof of Theorem~\ref{thm:SAEW}.     In the rest of the proof, we consider that \eqref{eq:induction} are satisfied for all $i\geq 0$. This occurs with probability $1- \delta$ and all inequalities stated in the proof of Theorem~\ref{thm:SAEW} are satisfied.

The proof is based on the following Lemma that we substitute to Inequality~\eqref{eq:errt} from the proof of Theorem~\ref{thm:SAEW}.

\begin{lemma}
\label{lem:errt2}
For all $t \in [t_i,t_{i+1}-1]$, with probability $1-\delta_{i+1}$,
\begin{equation*}
    \Err_{t-1} \leq 2\sqrt{2} X \sigma a_i' \sqrt{t-t_i} +  B c_i'\,,
\end{equation*}
where 
  $
      c_i' :=  b_i' + a_i'\big(\sqrt{\log \delta_{i+1}^{-1} } + \sqrt{2b} + 2 a\big)
  $. We recall that $\Err_{t-1}$ is defined in~\eqref{eq:risk}.
\end{lemma}
\begin{proof}[Proof of Lemma~\ref{lem:errt2}]
In the particular case of the square loss, the gradients are given by $\nabla \ell_t(\theta) = 2X_t(X_t^\top \theta - Y_t)$, so that 
\begin{equation}
    \|\nabla\ell_t(\hat \theta_{t-1})\|_\infty^2 \leq 4X^2 \ell_t(\hat \theta_{t-1}) \,.
    \label{eq:gradient_ellt}
\end{equation}
Following \cite[Corollary 2.2]{Gerchinovitz2011a}, we get from Inequality~\eqref{eq:reg_online_optimization} that
\[
   \sum_{t=t_{\mathrm{start}}}^{t_{\mathrm{end}}}  \ell_t(\hat \theta_{t-1}) -   \ell_t(\theta^\ast) \leq 2 a U X \sqrt{\sum_{t_{\mathrm{start}}}^{t_{\mathrm{end}}} \ell_t(\hat \theta_{t-1})} \\
      + b U B\,.
\]
Solving the second-order inequality (see~\cite[Lemma 10]{GaillardStoltzEtAl2014}), it yields the improvement for small losses
\[
    \sqrt{\sum_{t=t_{\mathrm{start}}}^{t_{\mathrm{end}}}  \ell_t(\hat \theta_{t-1})} \leq \sqrt{\sum_{t=t_{\mathrm{start}}}^{t_{\mathrm{end}}}  \ell_t(\theta^\ast)} + \sqrt{b U  B} + 2a U X\,.
\]
Thus, from~\eqref{eq:gradient_ellt},
\begin{multline*}
    \sqrt{\sum_{s=t_i}^{t-1} \big\|\nabla\ell_s(\hat \theta_{s-1})\big\|_\infty^2 } \leq 2X \sqrt{\sum_{s=t_{i}}^{t-1}  \ell_s(\theta^\ast)} \\
    + 2X \sqrt{b U  B} + 4 a U X^2  \,.
\end{multline*}
But, with probability $1-\delta_{i+1}$, we have from Theorem~\ref{thm:martingale}
\begin{align*}
    \sum_{s=t_{i}}^{t-1}  \ell_s(\theta^\ast) &\leq (e-1) \sum_{s=t_i}^{t-1} \E[\ell_s(\theta^\ast)] + (Y+XU)^2 \log \delta_{i+1}^{-1} \\
        &  \leq 2 \sigma^2 (t-t_i) + (Y+XU)^2 \log \delta_{i+1}^{-1}  \,,
\end{align*}
where $\sigma^2 = \E[\ell_t(\theta^\ast)]$. Plugging into the previous inequality and using $\sqrt{x+y}\leq \sqrt{x}+\sqrt{y}$ for $x,y>0$, this yields
\begin{align}
    & 2^{-1} X^{-1} \sqrt{\sum_{s=t_i}^{t-1} \big\|\nabla\ell_s(\hat \theta_{s-1})\big\|_\infty^2 } \\ 
      & \leq \sqrt{2} \sigma \sqrt{t-t_i} + (Y+XU) \sqrt{\log \delta_{i+1}^{-1}}  + \sqrt{b U  B} + 2 a U X \nonumber \\
      & \leq  \sqrt{2}\sigma \sqrt{t-t_i} + 2^{-1}BX^{-1} \big(\sqrt{\log \delta_{i+1}^{-1}}  + \sqrt{2b} + 2a \big) \,, \label{eq:bound_gradients_2}
\end{align}
where the second inequality is because $B/(2X) \geq (Y+XU) \geq XU$.
The proof of Lemma~\ref{lem:errt2} is concluded by using the definition of $\Err_{t-1}$ (see~\eqref{eq:risk}).
\end{proof}

The proof of Theorem~\ref{thm:squareloss} is then completed following the one of Theorem~\ref{thm:SAEW} by using Lemma~\ref{lem:errt2} instead of Inequality~\eqref{eq:errt}. Finally, it only suffices to substitute $B a_i'$ with $2\sqrt{2}X\sigma a_i'$ and $b_i'$ with $c_i'$ in the final results. At the end, $b'$ of Theorem~\ref{thm:SAEW} must thus be substituted with
\begin{align*}
    c' & := b' + a'\big(\sqrt{2 \log (1 +2 \log_2 T) - \log \delta } 
      + \sqrt{2b} + 4a\big) \\
      & \leq 1/2 + b + 3 \log(1+3\log T) - \log \delta  \\
      & \qquad + \big(a + \sqrt{6 \log(1 + 3 \log t) - 2 \log \delta} \big) \\
      & \pushright{ \big(\sqrt{2 \log (1 +3\log T) - \log \delta } + \sqrt{2b} + 2a\big) }\\
      & \leq \frac{1}{2} + b + 3 \log(1+3\log T) - \log \delta + 4 a^2 \\
      & \pushright{  + 2 b + 6 \log (1 + 3\log T) - 2 \log \delta) }\\
      & \leq  1/2 + 3b + 4a^2 + 9 \log (1 + 3\log T) - 3 \log \delta \,. \\
      & \lesssim 1+b+a^2 + \log\log T - \log \delta
\end{align*}

However, in contrast to the bound $B$ on the gradients, Lemma~\ref{lem:errt2} only holds with probability $1-\delta_{i+1}$ (instead of almost surely). A union bound over all events states that the final result only holds with probability $1-\delta - \sum_{i=1}^\infty \delta_{i+1} = 1-2\delta$. To get a result with probability $1-\delta$, $\delta$ must thus be multiplied by $2$ in the results.

This gives that, from the risk bound of Theorem~\ref{thm:SAEW}, with probability $1-\delta$, $\Risk\big(\tilde \theta_t\big)$ is upper-bounded by
\begin{align*}
        \min \bigg\{ & 4U \left(\frac{ X\sigma a'}{\sqrt{T}} + \frac{ B c '}{T}\right) + \frac{\alpha U^2}{8d_0T}, \\
        & \qquad {\frac{d_0 }{\alpha} \left(\frac{2^{10} X^2 \sigma^2 a'^2}{T} + \frac{2^{11}B^2 c'^2}{T^2}\right) + \frac{2\alpha U^2}{d_0 T^2}
        \bigg\}\,,}
\end{align*}
where $a' = 2a + 2\sqrt{6 \log(1 + 3 \log T) + 2 \log (2/\delta) }$ and $c' = 1 + 3b + 4a^2 + 9 \log (1 + 3\log T) + 3 \log (2/\delta) $.

The bound of the theorem is then obtained by using that $B = 2X(Y+2XU)$.

\subsection{Proof of Theorem~\ref{thm:calibration}}
\label{proof:calibration}

For the sake of clarity, we only perform this proof up to universal constants.
Let $B^\ast = 2X(Y+2X\|\theta^\ast\|_1) \geq \max_{\theta \in \cB(0,2\|\theta^\ast\|_1)} \|\nabla \ell_t(\theta)\|_\infty$ almost surely. We also define by $\alpha^\ast$ the maximal number strong convexity parameter that satisfies~\eqref{ass:strong_convexity}.

Let $T \geq 1$. Then, by definition (see Alg.~\ref{alg:calibration}), $\tilde f_{T-1} = \bar f_j$ for $j = \lfloor \log_2 T \rfloor-1$. 

We aim at controlling the excess risk of the average estimator $\bar f_j = \sum_{t=2^j}^{2^{j+1}-1} \hat f_t$. To do so, we control the cumulative risk for $t = 2^j,\dots,2^{j+1}-1$
\begin{multline*}
    \Risk^{(j)} := \sum_{t=2^j}^{2^{j+1}-1} \E_{t-1}\big[(Y_t-\hat f_t(X_t))^2\big] \\
      - \E\big[(Y_t-X_t^\top \theta^\ast)^2\big] \,,
\end{multline*}
where $\E_{t-1}[\,\cdot\,] = \E\big[\cdot|(X_1,Y_1),\dots,(X_{t-1},Y_{t-1})\big]$. We will use that 
\begin{equation}
    \label{eq:riskbarfj}
    \Risk(\bar f_j) \leq \Risk^{(j)}2^{-j} \lesssim \frac{\Risk^{(j)}}{T}.
\end{equation}
We first prove that it exists a predictor $f_{p,j}$ with $p \in \cG_j$ that has a small excess risk. Then, we will apply Theorem~4.5 of \cite{Wintenberger2014} to show that BOA almost achieves this performance.

\paragraph{Step~1. Either it exists a predictor $f_{p,j}$ with small excess risk or $\Risk^{(j)}$ is small.} 
Since all predictions $\hat f_t(X_t)$ lie in $[-Y,Y]$ almost surely, 
\begin{equation}
  \label{eq:riskjM}
    \Risk^{(j)} \leq Y^2 2^j \leq Y^2 T  \,.
\end{equation}
Let $d_0$ in $\cG_j$ (i.e., a power of 2) such that $d_0 / 2 \leq \|\theta^\ast\|_0 \leq d_0$. We show that if the conditions of Theorem~\ref{thm:squareloss} cannot be satisfied with any parameter of the grid $\cG_j$, the cumulative risk $\smash{\Risk^{(j)}}$ is small enough. We start with the choice of the parameter $U$, which should be of order $\|\theta^\ast\|_1$:
\begin{enumerate}[label=\emph{\alph*})]
  \item If $\|\theta^\ast\|_1 \leq 2^{-2j}$. It exists a predictor in $\cG_j$ such that $f_{p,j} = 0$ (consider $d_0 =0$). In this case, 
  \begin{multline*}
      \Risk(f_{p,j}) = \E[(Y_t-0)^2] \leq B^\ast \|\theta^\ast\|_1 \\
       \leq B^\ast 2^{-2j} \lesssim B^\ast T^{-2} \,,
  \end{multline*}
  where we used that $2^{-j}\lesssim  T^{-1}$.
  \item If $\|\theta^\ast\|_1 \geq 2^{2j + \lceil 2 \log Y\rceil}$, then $2^j \leq Y^{-2}\|\theta^\ast\|_1 2^{-j}$ and from Inequality~\eqref{eq:riskjM}, 
    \[
        \Risk^{(j)} \leq  \|\theta^\ast\|_1 2^{-j} \lesssim \frac{\|\theta^\ast\|_1}{T} \lesssim \frac{\|\theta^\ast\|_0 (B^\ast)^2}{\alpha^\ast T} \,.
    \]
\end{enumerate}
Otherwise, we can choose $U$ in $\cG_j$ such that $U/2 \leq \|\theta^\ast\|_1 \leq U$. Similarly for $B$:
\begin{enumerate}[label=\emph{\alph*}),start=3]
  \item if $B < 2^{-2j}$, then for $f_{p,j} = 0$,
  \begin{multline*}
      \Risk(f_{p,j}) = \E[\ell(Y_t,0)] \leq B^\ast \|\theta^\ast\|_1 \\
         \leq \|\theta^\ast\|_1 2^{-2j}\lesssim \frac{\|\theta^\ast\|_1 }{T^2}  \,,
  \end{multline*}
  \item if $B > 2^{2j+ \lceil 2\log Y\rceil }$, then from Inequality~\eqref{eq:riskjM}, $\Risk^{(j)} \leq B^\ast 2^{-j} \lesssim B^\ast T^{-1}$.
\end{enumerate}
Otherwise, we can choose $B$ in $\cG_j$ such that $B/2 \leq B^\ast \leq B$. Finally, for $\alpha$:
\begin{enumerate}[label=\emph{\alph*}),start=5]
   \item if $\alpha^\ast < 2^{-2j + \lceil \log_2 (B^2d_0/Y^2)\rceil} \leq d_0 B^2 2^{-2j}/Y^2$, then $2^j \leq d_0 B^2 / (Y^2 \alpha^\ast 2^j)$ and thus
  \[
      \Risk^{(j)} \leq Y^2 2^j \leq Y^2\frac{d_0 B^2}{Y^2 \alpha^\ast 2^j} \lesssim \frac{\|\theta^\ast\|_0 (B^\ast)^2}{ \alpha^\ast T}\,.
  \]
\end{enumerate}

Otherwise, we can choose $\alpha$ in $\cG_j$ such that $\min\{d_0/T,\alpha^\ast/2\}\leq  \alpha \leq \alpha^\ast$. 
\begin{enumerate}[label=\emph{\alph*}),start=6]
    \item Applying Theorem~\ref{thm:squareloss}, with high probability the excess risk of the estimator $f_{p,j}$ with the choice $(d_0,\alpha,U,B)$ described above satisfies
  \begin{align*}
   & \hspace*{-1cm} \Risk(f_{p,j}) 
         \stackrel{\text{clipping}}{\leq} \Risk(\tilde \theta_{p,j})\\
        & \hspace*{-1cm} \lesssim    \min \bigg\{\frac{ X^2 }{\gamma} \left(\frac{\sigma^2 a'^2}{T} + \frac{(Y+X\|\theta^\ast\|_1)^2 c'^2}{T^2}\right) + \frac{\gamma \|\theta^\ast\|_1^2}{ T^2} , \\
        & \hspace*{-1cm}  \qquad  \|\theta^\ast\|_1 X \left(\frac{ \sigma a'}{\sqrt{T}} + \frac{ (Y+X\|\theta^\ast\|_1) c '}{T}\right) + \frac{\gamma \|\theta^\ast\|_1^2}{ T}
        \bigg\}\,,
  \end{align*}
  with $\gamma = \max\{d_0/\alpha,1/T\}$.
\end{enumerate}

Putting everything together, either (for cases b), d), and e))
  \begin{equation}
    \label{eq:riskj1}
       \Risk^{(j)} \lesssim \Big(B^\ast  + \frac{\|\theta^\ast\|_0 (B^\ast)^2}{\alpha^\ast}\Big) T^{-1}
  \end{equation}
  or, for cases a), c), and f), there exists $p \in \cG_j$ such that with high probability
  \begin{align}
      &  \Risk(f_{p,j}) \lesssim   
         \min \bigg\{
            \frac{1}{\gamma} \left(\frac{X^2\sigma^2 a'^2}{T} + \frac{(B^\ast c')^2}{T^2}\right) + \frac{\gamma \|\theta^\ast\|_1^2}{ T^2} , \nonumber \\
        &    \|\theta^\ast\|_1 X \left(\frac{ \sigma a'}{\sqrt{T}} + \frac{ (Y+X\|\theta^\ast\|_1) c '}{T}\right) + \frac{\gamma \|\theta^\ast\|_1^2}{ T} \bigg\} + \frac{B^\ast}{T^2}
        \,, \label{eq:riskj2}
  \end{align}
  
  \paragraph{Step~2. Bound of the meta-algorithm.} Using that the square loss is $4Y$-Lipschitz over the domain $[-2Y,2Y]$ and 2-strongly convex, we can apply Theorem~4.5 of \cite{Wintenberger2014} with $C_b = 4Y$, $C_\ell = 2$, and $M = \# \cG_j$.  We get that with high enough probability 
  \begin{align*}
    \Risk^{(j)}  \lesssim &\  T \min_{p \in \cG_j} \Risk(f_{p,j}) \\
      &  \pushright{+ Y^2 \big(\log \# \cG_j + \log(\log T+\log Y)  - \log \delta \big) \,.}
  \end{align*}
  Substituting 
    \begin{equation*}
        \#\cG_j \lesssim (j+\log Y)^3 \log d 
            \lesssim (\log T + \log Y)^3 \log d \,,
    \end{equation*}
  this yields
  \begin{align*}
    \Risk^{(j)}  \lesssim &\  \frac{\min_{p \in \cG_j} \Risk(f_{p,j})}{T} \\
      &  \pushright{+ Y^2 \big(\log \log d + \log(\log T+\log Y)  - \log \delta \big) \,.}
  \end{align*}
  Combining with Inequality~\eqref{eq:riskj2}, we obtain that $\Risk^{(j)}$ is at most of order
  \begin{align*}
    &  \Risk^{(j)} \lesssim   Y^2 \big(\log \log d + \log(\log T+\log Y)  - \log \delta \big)  \\
    & +  \min \bigg\{
            \frac{1}{\gamma} \left(X^2\sigma^2 a'^2 + \frac{(B^\ast c')^2}{T}\right) + \frac{\gamma \|\theta^\ast\|_1^2}{ T} , \nonumber \\
        &    \|\theta^\ast\|_1 X \left( \sigma a'\sqrt{T} +  (Y+X\|\theta^\ast\|_1) c '\right) + \gamma \|\theta^\ast\|_1^2 \bigg\} + \frac{B^\ast}{T}.
  \end{align*}
  Finally, using Inequality~\eqref{eq:riskbarfj}, keeping only the main asymptotic term in $1/T$, and substituting $a'\lesssim \log ((d\log T)/\delta)$ concludes the proof.

%%% 
\section{Martingale inequalities}
In this section, we prove two martingale inequalities that are used in the analysis.

\subsection{Poissonian inequality}

First, we prove a Poissonian inequality which only works for nonnegative increments.
\begin{theorem}
  \label{thm:martingale}
   Let $T \geq 1$. Let $(X_t)_{t\geq 1}$ be a sequence of random variables such that $X_t \in [0,B]$ almost surely, then with probability at least $1-\delta$
    \[
      \sum_{t=1}^T X_t \leq  (e-1)\sum_{t=1}^T \E_{t-1}[X_t] + B\log(1/\delta)\,.
    \]
\end{theorem}
\begin{proof}
Let $Z_t = X_t / B \in [0,1]$. From \cite[Lemma A.3]{Cesa-BianchiLugosi2006}, for all $t \geq 1$, and all $s>0$
\[
    \E_{t-1}\Big[\exp\big(sZ_{t} - (e^s-1) \E_{t-1}[Z_t]\big)\Big] \leq 1 \,.
\]
Thus,
\begin{align*}
    \E\bigg[& \exp\Big(s\sum_{t=1}^T Z_t - (e^s-1)\sum_{t=1}^T\E_{t-1}[Z_t]\Big)\bigg] \\
      & = \E\bigg[ \E_{T-1} \Big[\exp\big(sZ_T - (e^s-1)\E_{T-1}[Z_T]\big) \Big] \\
      & \pushright{\exp\Big(s\sum_{t=1}^{T-1} Z_t - (e^s-1)\sum_{t=1}^{T-1}\E_{t-1}[Z_t]\Big)\bigg]} \\
      & \leq \E\bigg[\exp\Big(s\sum_{t=1}^{T-1} Z_t -  (e^s-1)\sum_{t=1}^{T-1}\E_{t-1}[Z_t]\Big)\bigg]
\end{align*}
By induction, we get 
\[
  \E\bigg[\exp\Big(s\sum_{t=1}^T Z_t -  (e^s-1)\sum_{t=1}^T\E_{t-1}[Z_t]\Big)\bigg] \leq 1 \,.
\]
We conclude thanks to Markov's inequality, with probability at least $1-\delta$
\[
   \sum_{t=1}^T Z_t \leq   \frac{e^s-1}{s} \sum_{t=1}^T \E_{t-1}[Z_t] + \frac{1}{s}\log(1/\delta)  \,.
\]
The final result is obtained by substituting $Z_t = X_t/B$ and by choosing $s=1$.
\end{proof}

\subsection{From cumulative regret to cumulative risk}
\label{sec:riskregret}
\begin{theorem}
  \label{thm:riskregret}
  Let $x>0$. Assume $\theta^\ast \in \cB_1(\theta_{\mathrm{center}},\epsilon)$. The cumulative risk of any convex optimization procedure in $\cB_1(\theta_{\mathrm{center}},\epsilon)$ satisfies, with probability $1-\delta$
    \begin{align*}
      \Risk_{1:T}
      & (\hat \theta_{0:(T-1)})  - \Reg_{1:T}(\hat \theta_{0:(T-1)}) \\
      & \leq  \epsilon \sqrt{2 \log\Big(\frac{2 + \log(T/2)}{2\delta}\Big)  \sum_{t=1}^T \big\|\nabla \ell_t(\hat \theta_{t-1})\|_\infty^2} \\
      & \pushright{+ \Big(\frac{1}{2} + \log\big(1 + \frac{1}{2}\log(T/2)\big) - \log \delta \Big) \epsilon B} \,,
    \end{align*}
  where $B \geq \max_{\theta \in \cB_1(\theta_{\mathrm{center}},\epsilon)} \big\|\nabla \ell_t(\hat \theta_{t-1})\|_\infty$ almost surely.
\end{theorem}

\begin{proof}
This is a consequence of Theorem 4.1 of \cite{Wintenberger2014}. Let $\delta \in (0,1)$ and  $(\eta_t)_{t\geq 0}$ be a sequence adapted to the filtration $(\cF_{t} = \{\ell_1,\dots,\ell_{t-1}\})_{t\geq0}$. Then, with the notation $\smash{\ell_{j,t}^2 \leq \epsilon^2 \|\nabla \ell_t(\hat \theta_{t-1})\|_\infty^2}$, applying Theorem 4.1 of \citet{Wintenberger2014}, we get that with probability $1-\delta$
   \begin{multline}
      R_T := \Risk_{1:T}(\hat \theta_{0:(T-1)}) - \Reg_{1:T}(\hat \theta_{0:(T-1)})  \\
       \leq   \epsilon^2 \sum_{t=1}^T \eta_{t-1} \|\nabla \ell_t(\hat \theta_{t-1})\|^2_\infty \\
        +  \frac{\log\Big(1+\E\big[\log(\eta_1/\eta_T)\big]\Big) - \log \delta}{\eta_T}\,,
        \label{eq:regret2risk}
   \end{multline}
   where $\Risk_{1:T}(\hat \theta_{0:(T-1)})  := \sum_{t=1}^T \E[\ell_t](\hat \theta_{t-1}) - \E[\ell_t](\theta^\ast)$ and $\Reg_{1:T}(\hat \theta_{0:(T-1)})  := \sum_{t=1}^T  \ell_t(\hat \theta_{t-1}) - \ell_t(\theta^\ast)$.
  
    We obtain the stated inequality from~\eqref{eq:regret2risk}, by properly setting the tuning parameters 
    \[
      \eta_{t} := \frac{1}{\epsilon} \min\left\{\frac{1}{B} , \frac{c \Gamma}{V_{t-1}}\right\}\,,
    \]
    where $c$ will be set by the analysis and
    \[
    \Gamma := \sqrt{\log\big(1+\log(\sqrt{T}/c)\big) - \log \delta}\,,
    \] 
    and
    \[
    V_{t-1} := \sqrt{\sum_{s=1}^{t-1} \big\|\nabla\ell_s(\hat \theta_{s-1})\big\|_\infty^2}\,.
    \]
    Indeed, first we use that that $\eta_1/\eta_T \leq \sqrt{T}/{c}$ so that $\E[\log (\eta_1/\eta_T)] \leq \log (\sqrt{T}/c)$. Then, similarly to the proof of \cite[Theorem 5]{CesaBianchiMansourStoltz2007}, we can show that the first term in the right-hand side of~\eqref{eq:regret2risk} is upper-bounded as
  \[
      \sum_{t=1}^T \eta_{t-1} \|\nabla \ell_t(\hat \theta_{t-1})\|^2_\infty \leq \frac{B}{2\epsilon} + \frac{c\Gamma}{2\epsilon}V_T.
  \]
  But, by definition of $\eta_T$, the second term is also controlled as
  \[
     \frac{\log\Big(1+\E\big[\log(\eta_1/\eta_T)\big]\Big) - \log \delta}{\eta_T} \leq \epsilon \Gamma  \max \left\{  B \Gamma, \frac{1}{c} V_T\right\}\,.
  \]
  Plugging these two last inequalities into~\eqref{eq:regret2risk} leads to
   \begin{equation*}
      R_T
      \leq \frac{B\epsilon}{2} + \frac{c \Gamma \epsilon}{2}V_T +  \epsilon \Gamma  \max \left\{  B \Gamma, \frac{V_T}{c} \right\}\,.
   \end{equation*}
   We then need to distinguish two cases 
   \begin{itemize}
    \item if $c \Gamma B \leq V_T$, then optimizing in $c = \sqrt{2}$
    \begin{equation*}
      R_T
      \leq \frac{B\epsilon}{2} + \Big(\frac{c}{2} + \frac{1}{c}\Big) \Gamma \epsilon V_T \leq  \frac{B\epsilon}{2} + \sqrt{2} \Gamma \epsilon V_T
    \end{equation*}
    \item if $c \Gamma B \geq V_T$, then
    \[
      R_T \leq \frac{B\epsilon}{2} + \frac{1}{\sqrt{2}} \Gamma \epsilon V_T + \epsilon B \Gamma^2 \,.
    \]
   \end{itemize}
   Therefore, putting the two cases together
   \[
      R_T \leq \frac{B\epsilon}{2} + \sqrt{2} \Gamma \epsilon V_T + \epsilon B \Gamma^2 \,.
   \]
   We conclude the proof by substituting $\Gamma$ and $V_T$ with their definitions.
\end{proof}

\end{document}